%% file: main.tex
\documentclass[reqno]{amsart}
\usepackage{amsmath}
\usepackage{amssymb, amsthm, amsfonts, tikz-cd, mathrsfs,mathtools,stmaryrd,enumitem, algorithmic,float,comment,tikz,hyperref}
\usepackage[toc]{appendix}
\usetikzlibrary{backgrounds}
\usetikzlibrary{decorations.pathreplacing}
\usepackage{fullpage}
\usepackage{xcolor}

\usepackage{tikz}
\usetikzlibrary{calc,decorations.pathreplacing}

\hypersetup{
  colorlinks   = true, 
  urlcolor     = blue, 
  linkcolor    = blue, 
  citecolor   = red 
}

\usepackage{comment}

\usepackage[capitalize,nameinlink]{cleveref}
\crefname{theorem}{Theorem}{Theorems}
\crefname{lemma}{Lemma}{Lemmas}
\crefname{claim}{Claim}{Claims}
\crefname{prop}{Proposition}{Propositions}
\crefname{figure}{Figure}{Figures}

\newtheorem{theorem}{Theorem}

\newtheorem{lemma}[theorem]{Lemma}
\newtheorem{claim}[theorem]{Claim}
\newtheorem{corollary}[theorem]{Corollary}

\newtheorem{prop}[theorem]{Proposition}

\newtheorem*{claim*}{Claim}

\theoremstyle{remark}
\newtheorem*{remark*}{Remark}

\theoremstyle{definition}
\newtheorem{definition}[theorem]{Definition}

\numberwithin{theorem}{section}

\renewcommand{\phi}{\varphi}

\newcommand{\eps}{\varepsilon}

\newcommand{\cE}{\mathcal E}

\newcommand{\cH}{\mathcal H}

\newcommand{\E}{\mathbb{E}}

\newcommand{\Var}{\operatorname{Var}}
\def\1{\mathbbm{1}}

\renewcommand{\le}{\leqslant}
\renewcommand{\ge}{\geqslant}
\renewcommand{\P}{\mathbb{P}}

\newcommand{\R}{\mathbb R}

\newcommand{\ceil}[1]{\lceil#1\rceil}
\newcommand{\floor}[1]{\lfloor#1\rfloor}
\newcommand{\ang}[1]{\langle#1\rangle}

\newcommand{\bin}{\{\pm 1\}}

\newcommand{\var}{\textup{Var}}

\newcommand{\unif}{U([-1,1])}

\usepackage[
backend=biber,
style=numeric,
maxnames=99,
giveninits=true
]{biblatex}
\renewbibmacro{in:}{%
  \ifentrytype{article}{}{\printtext{\bibstring{in}\intitlepunct}}}

\title
{Improved Lower bound for hypercube edge slicing} 

\author{Lisa Sauermann}
\thanks{University of Bonn, Germany. Email: \texttt{sauermann@iam.uni-bonn.de}. Supported by the DFG Heisenberg Program.}

\author{Zixuan Xu}
\thanks{Massachusetts Institute of Technology. Email: \texttt{zixuanxu@mit.edu}}

\begin{document}

\begin{abstract}
How many hyperplanes in $\R^n$ are needed in order to slice every edge of the $n$-dimensional hypercube with vertex set $\{\pm 1\}^n$? Here, we say that a hyperplane $H\subseteq \R^n$ slices an edge of the hypercube if it contains exactly one interior point of the edge. The problem of determining the minimum possible size of a collection of hyperplanes in $\R^n$, such that every edge of the hypercube is sliced by at least one of these hyperplanes, is more than 50 years old and has been studied by many researchers. We prove that, for sufficiently large $n$, at least $\Omega(n^{13/19}\log^{-32/19}n)$ hyperplanes are needed, improving upon the best previous lower bound $\Omega(n^{2/3}\log^{-4/3}n)$ due to Klein.
\end{abstract}
		
\maketitle

\vspace{-2em}

\section{Introduction}

The $n$-dimensional hypercube has vertex set $\bin^n$ and edges between pairs of vertices differing in exactly one coordinate.
We say that a hyperplane $H\subseteq \R^n$ slices an edge of the $n$-dimensional hypercube if $H$ contains exactly one interior point of the edge. In other words, letting $vv'$ be an edge with endpoints $v,v'\in \bin^n$ and $H$ be a hyperplane defined as $H = \{x\in \R^n\mid \ang{a,x} = b \}$, the hyperplane $H$ is said to slice the edge $vv'$ if $\ang{a,v} - b$ and $\ang{a,v'} - b$ are both nonzero and have different signs.

It is natural to ask how many hyperplanes in $\R^n$ are needed in order to slice every edge of the $n$-dimensional hypercube. More formally, what is the minimum possible size of a hyperplane collection $\cH$ in $\R^n$ such that every edge of the $n$-dimensional hypercube $\bin^n$ is sliced by at least one hyperplane in $\cH$? This problem has attracted the attention of many researchers over the past 50 years \cite{ONeil71, grunbaum72, EmamyKhansary86,AHLSWEDE1990137,saks1993,GotsmanLinial94}, and has applications to the study of perceptrons \cite{ONeil71} and of threshold circuits for parity \cite{paturi1990threshold,yehuda2021slicing}. 

It is conjectured that the minimum possible size of a hyperplane collection $\cH$ slicing every edge of the $n$-dimensional hypercube is $\Theta(n)$, and the current best known upper bound is $\ceil{(5/6)n}$ given by a construction of Paterson (see \cite{saks1993}). So far, this conjecture is only known to be true under some very specific assumptions on the form of the hyperplanes in $\cH$. It is not hard to show that at least $n$ hyperplanes are needed if for the hyperplanes one can choose normal vectors all of whose entries are non-negative, see \cite{AHLSWEDE1990137} or \cite{GotsmanLinial94}. The conjecture is also known to be true in the case where all hyperplanes in $\cH$ have normal vectors in $\{1,-1\}^n$ following a result of Alon--Bergmann--Coppersmith--Odlyzko~\cite{alon88balance}. Recently, it was shown in \cite{sauermann2025nondegeneratehyperplanecovershypercube} that the conjecture also holds when all hyperplanes have normal vectors in $\{-C, \dots, C\}^n$ for a fixed integer $C>0$.

In the general setting, O'Neil~\cite{ONeil71} showed that any collection $\cH$ of hyperplanes slicing every edge of the hypercube must have size $|\cH|\ge \Omega(\sqrt{n})$. Yehuda and Yehudayoff~\cite{yehuda2021slicing} improved this lower bound to $|\cH|\ge \Omega(n^{0.51})$. A few years ago, Klein~\cite{klein2022slicing} showed $|\cH|\ge \Omega(n^{2/3}\log^{-4/3}n)$, which has been the best known lower bound.

In this paper, we improve upon this, showing the following new lower bound.

\begin{theorem}\label{thm:main}
    Let $\cH$ be a collection of hyperplanes such that every edge of the $n$-dimensional hypercube $\bin^n$ is sliced by some hyperplane in $\cH$. For $n$ sufficiently large, we have
    \[|\cH| > \frac{n^{13/19}}{10^{11}\log^{32/19}n}.\]
\end{theorem}

\cref{thm:main} also implies a new lower bound for the number of wires in depth-two threshold circuits computing the parity function, see \cite{yehuda2021slicing} for more details.

\subsection{Paper organization} In \cref{sec:overview}, we give a high-level overview of our proof highlighting the key steps and intuition. In \cref{sec:prelim}, we present definitions and probabilistic lemmas that will be used in our proof. Then in \cref{sec:main-proof}, we give the proof of \cref{thm:main} assuming a key proposition (\cref{prop:construct-X}). Finally in \cref{sec:prop-proof}, we provide the proof of \cref{prop:construct-X}, completing the proof of \cref{thm:main}.

\section{Proof Overview}\label{sec:overview}

We start by outlining the main proof strategy, highlighting the key components of the proof. Our proof combines ideas from \cite{yehuda2021slicing} and \cite{klein2022slicing} together with several new ideas.

Let $\cH = \{H_1,\dots, H_k\}$ be the collection of hyperplanes, where $H_i = \{x\in \R^n \mid \ang{v_i, x} = \lambda_i\}$ has normal vector $v_i = (v_{i1}, \dots, v_{in})\in \R^n$ for each $i\in [k]$. Note that rescaling the coefficients of the hyperplane equation does not change the slicing properties of the hyperplane, so we may further assume without loss of generality that $||v_i||_2 = 1$ for all $i\in [k]$. Then our goal is to find a point $X\in [-1/2,1/2]^n$ that is far away from most of the hyperplanes. To be slightly more precise, we want to find $X\in [-1/2,1/2]^n$ such that 
\[|\ang{X, v_i} - \lambda_i| > 5\sqrt{\log n}\]
holds for all but at most $O(\sqrt{n})$ indices $i\in [k]$. We can then randomly ``round '' $X\in [-1/2,1/2]^n$ to a vertex $z\in \bin^n$ (to be precise, we take the entries $z_1,\dots,z_n\in \bin$ of $z$ to be independent, with distributions chosen in such a way that $\E[z_i]=X_i$ for each $i$), and consider a uniformly random edge adjacent to $z$. We claim that with positive probability this edge is not sliced by any of the hyperplanes in $\cH$. 

Indeed, with high probability, ``rounding'' $X\in [-1/2,1/2]^n$ to a vertex $z\in \bin^n$ changes the value of $\ang{X,v_i}$ by no more than $3\sqrt{\log n}$ for each $i\in [k]$. So with high probability $z\in \bin^n$ is at distance at least $2\sqrt{\log n}$ from all but at most $O(\sqrt{n})$ hyperplanes in $\cH$ (more precisely, we have $|\ang{z, v_i} - \lambda_i| > 2\sqrt{\log n}$ for all of the $k-O(\sqrt{n})$ hyperplanes $H_i\in \cH$ with $|\ang{X, v_i} - \lambda_i| > 5\sqrt{\log n}$). A hyperplane $H_i\in \cH$ at distance at least $2\sqrt{\log n}$ from $z$ cannot slice any edge adjacent to $z$, because any adjacent vertex $z'\in \bin^n$ must satisfy $|\ang{z,v_i} - \ang{z', v_i}|\le 2$, so $\ang{z,v_i} - \lambda_i$ and $\ang{z', v_i}-\lambda_i$ cannot have different signs. So the only hyperplanes in $\cH$ possibly slicing any edge adjacent to $z$ are the ones that are close to $X$. Using the Erd\H{o}s--Littlewood--Offord theorem, one can show that a given hyperplane slices a uniformly random edge adjacent to the random rounding $z$ of a given point $X\in [-1/2,1/2]^n$ with probability at most $O(1/\sqrt{n})$. Since there are at most $O(\sqrt{n})$ hyperplanes close to $X$, picking appropriate constants, a union bound implies that with positive probability a random edge adjacent to $z$ is not sliced by any of the hyperplanes in $\cH$. The details of this part are given in \cref{sec:main-proof}.

So far, the strategy described above was employed in both \cite{yehuda2021slicing} and \cite{klein2022slicing}. The key part of the problem is to construct a point $X\in [-1/2,1/2]^n$ that is close to only few of the hyperplanes in $\cH$. For simplicity, let us assume that for each $i\in [k]$, the hyperplane normal vector $v_i$ satisfies $||v_i||_\infty \le 10/\sqrt{n}$. In \cite{yehuda2021slicing}, $X$ was chosen by applying Bang's lemma (originating from his solution \cite{bang1951} to Tarski’s plank problem) to the $k\times n$ matrix with rows $v_1,v_2,\dots,v_n$. On the other hand, in \cite{klein2022slicing}, $X$ is taken to be a random linear combination of the vectors $v_1,v_2,\dots,v_n$. More specifically, 
\[X = 10^{-2}\sqrt{\frac{n}{k\log n}}\sum_{i\in [k]}\alpha_iv_i, \quad\quad  \alpha_i\sim U[-1,1]\text{ independent for all } i\in [k],\]
where $U[-1,1]$ denotes the uniform distribution over the real interval $[-1,1]$ and $\alpha_1,\dots, \alpha_k$ are independent. By our assumption that $||v_i||_\infty \le 10/\sqrt{n}$ for each $i\in [k]$ and the Chernoff bound, with high probability $X$ is contained in $[-1/2,1/2]^n$. For any $i\in [k]$, the probability that $|\ang{X,v_i} - \lambda_i|\le 5\sqrt{\log n}$ is at most $O(\sqrt{k}\log n/\sqrt{n})$. So the expected number of such indices $i$ is $O(k^{3/2}\log n/\sqrt{n})$. If $k = \Tilde{O}(n^{2/3})$ where $\Tilde{\Omega}(\cdot)$ ignores logarithm factors, then $O(k^{3/2}\log n/\sqrt{n}) \le  O(\sqrt{n})$. Thus, Klein's proof in \cite{klein2022slicing} gives a bound of $k = \Tilde{\Omega}(n^{2/3})$ for the number of hyperplanes in the collection $\cH$.

In our proof, we further develop Klein's strategy from~\cite{klein2022slicing}. We first define $X_0$ similarly as in Klein's~\cite{klein2022slicing} proof:
\[X_0 = \rho_0\sum_{i\in [k]}\alpha_iv_i, \quad\quad  \alpha_i\sim U[-1,1] \text{ independent for all } i\in [k]\]
for a suitably chosen parameter $\rho_0$. Then for an outcome of $X_0$, we say that $i\in [k]$ is \emph{bad} if $|\ang{X_0, v_i} - \lambda_i| \le 10\sqrt{\log n}$, i.e.\ if $H_i$ is ``close'' to $X_0$. We then take $X$ to be $X = X_0+X_1$ where 
\[X_1 = \rho_1 \sum_{\substack{i\in [k]\\i\text{ bad}}}\beta_i v_i,\quad\quad \beta_i\sim U[-1,1]\text{ independent for all bad }i\in [k]\]
for a suitably chosen parameter $\rho_1$. Intuitively, we sample $\beta_i$ for the bad indices $i\in [k]$ and add $\rho_1\beta_i v_i$ to $X_0$ to attempt to push $X_0$ away from the hyperplanes it is close to. However, the main difficulty comes from analyzing the expected number of hyperplanes close to $X$, since adding the term $X_1$ might push $X_0$ towards hyperplanes that were originally far away. The full details are given in \cref{sec:prop-proof}.

To remove the extra assumption that $||v_i||_\infty\le 10/\sqrt{n}$ for all $i\in [k]$, we use the matrix decomposition given in \cite{yehuda2021slicing} and the notion of vectors containing many scales introduced in \cite{yehuda2021slicing} (see \cref{sec:prelim}).

\section{Preliminaries}\label{sec:prelim}

\noindent \emph{Notation.} We write $[n] = \{1,\dots, n\}$. All logarithms are to base $e$ unless specified otherwise. We use $U([a,b])$ to denote the uniform distribution over the real interval $[a,b]$. For a vector $v\in \R^n$ and a subset $S\subseteq [n]$, we use $v|_S\in \R^S$ to denote the vector $v$ restricted to the subset of coordinates in $S$. For two sets $A$ and $B$, we use $A^B$ to denote the set of vectors indexed by elements in $B$ with entries in $A$. When writing $x\sim \bin^n$, we always mean $x$ is uniformly distributed in $\bin^n$.

\medskip

We first recall the useful notion of vectors containing many scales introduced in \cite{yehuda2021slicing}. For convenience, we will state the definitions and corresponding lemmas with explicit constants. The proofs of all the lemmas stated in this section are either standard or follow directly from \cite{yehuda2021slicing}, but for the reader's convenience, we collect them in \cref{app:probability-proofs}.

\begin{definition}[Distribution $\mu_p$]
    Let $p\in \R^n$ be a vector with $||p||_\infty\le 1$. We define the distribution $\mu_p$ over $x\in \bin^n$ by taking independent random entries $x_1,\dots, x_n\in \bin$ with distributions given by 
    \[x_i = \begin{cases} 1 & \text{with probability } (1+p_i)/2 \\ -1 & \text{with probability } (1-p_i)/2 \end{cases} \quad \text{for all } i\in [n].\]
\end{definition}

In particular, note that that for any $p\in \R^n$ with $||p||_\infty\le 1$, we have $\E_{x\sim \mu_p}[x_i] = p_i$ for all $i\in [n]$.

\begin{definition}[\cite{yehuda2021slicing}]\label{def:scales}
    For $\delta>0$ and a positive integer $s$, we say that a vector $v\in \R^n$ \emph{contains $s$ scales of size at least $\delta$} if there exists disjoint subsets $I_1,\dots, I_s\subseteq [n]$ such that $||v|_{I_s}||_2 \ge \delta$ and for $i\in [s-1]$ we have
    \[||v|_{I_i}||_2 \ge 100||v|_{I_{i+1}}||_2.\]
\end{definition}

The following lemma shows strong anticoncentration properties for vectors containing many scales.
\begin{lemma}[\cite{yehuda2021slicing}]\label{lem:many-scale-anticoncentration}
    For $\delta > 0$ and a positive integer $s\ge 100$, let $v\in \R^n$ be a vector containing $s$ scales of size at least $10\delta$. Then for any $b\in \R$ and uniformly random $x\in \bin^n$, we have 
    \[\P[|\ang{x,v} - b| \le \delta] \le e^{-s/100}.\]
\end{lemma}

We need the Chernoff--Hoeffding inequality.

\begin{lemma}[Chernoff--Hoeffding inequality~\cite{Hoeffding1963}]\label{lem:chernoff}
    Let $a_1,\dots,a_r,b_1,\dots,b_r$ be real numbers with $a_i<b_i$ for all $i\in [r]$, and let $X_1,\dots, X_r$ be independent random variables such that $a_i\le X_i\le b_i$ for all $i\in [r]$. Let $X = \sum_{i = 1}^r X_i$, then for all $t > 0$, we have
    \[\P[|X - \E[X]| \ge t] \le 2\exp\Big(-\frac{2t^2}{\sum_{i \in [r]}(b_i-a_i)^2}\Big).\]
\end{lemma}

We need the Erd\H{o}s--Littlewood--Offord theorem for biased linear forms.

\begin{theorem}[Erd\H{o}s--Littlewood--Offord]\label{thm:littlewood-offord}
    Let $v\in \R^n$ and $p\in \R^n$ with $||p||_\infty\le 1/2$. Let $t>0$ and let $m\ge 1$ be the number of indices $i\in [n]$ such that $|v_i|\ge t$. Then, for any $b\in \R$, for a random vector $x\sim \mu_p$, we have
    \[\P[|\ang{x,v} - b| < t] \le \frac{10}{\sqrt{m}}.\]
\end{theorem}

We also need a Littlewood--Offord type lemma for continuous random variables, which follows from a more general statement on the anticoncentration of sums of log-concave independent random variables due to Bobkov and Chistyakov~\cite{BobkovChistyakov2015-ConcentrationFunctions}.

\begin{lemma}\label{lem:continuous-littlewood-offord}
    Let $a_1,\dots, a_n, b_1,\dots, b_n$ be real numbers with $a_i<b_i$ for all $i\in [n]$, and let $X_i\sim U[a_i, b_i]$ be independent random variables for all $i\in [n]$. Define $X = \sum_{i = 1}^n X_i$. Then for any $t>0$ and $b\in \R$, we have
    \[\P[|X - b| < t]\le \frac{2t}{\sqrt{\var(X)}}.\]
\end{lemma}

\section{Proof of \cref{thm:main}}\label{sec:main-proof}

In this section, we present the proof of \cref{thm:main}. Note that we can assume without loss of generality that all hyperplanes are described by equations with nonzero coefficients. Indeed, given any collection $\cH$ of hyperplanes slicing all hypercube edges, one can first ``wiggle'' the hyperplanes so that all coefficients in the hyperplane equations become nonzero.

First, we recall a result on decomposing the coefficient matrix of the hyperplane collection introduced in \cite{yehuda2021slicing}. For matrices $A,A'\in \R^{k\times n}$ with rows $a_1,\dots, a_k$ and $a_1',\dots, a_k'$ respectively, we say that $A'$ is a \emph{row rescaling} of $A$ if there exists $\phi_1,\dots, \phi_k\in \R\setminus\{0\}$ such that $a_i' = \phi_i a_i$ for all $i\in [k]$.

\begin{prop}[Matrix Decomposition \protect{\cite[Lemma 12]{yehuda2021slicing}}]\label{prop:decomp}
Let $A$ be a real $k\times n$ matrix where all entries are  nonzero. There exists a partition of the row indices $[k] = K_1\sqcup K_2$ and a partition of the column indices $[n] = N_1\sqcup N_2$ with $|N_2|\le n/2$ such that for some row rescaling $A'$ of $A$ the following holds:
\begin{enumerate}
    \item For every $i\in [k]$, the $i$-th row vector $a'_i\in \R^n$ of $A'$ satisfies $||a'_i|_{N_1}||_2 = 1$.

    \item For every $j\in N_1$, the $j$-th column vector $a'_{*j}\in \R^k$ of $A'$ satisfies $||a'_{*j}|_{K_1}||_2\le  10^4(k\log n/n)^{1/2}$.

    \item For every $i\in K_2$, the restriction $a_i'|_{N_2}$ of the $i$-th row vector $a_i'$ of $A'$ to $N_2$ contains at least $\ceil{250\log n}$ scales of size at least $100$.
\end{enumerate}
\end{prop}

\begin{figure}
\centering
\begin{tikzpicture}
    
    \def\W{11.3}   
    \def\H{6}    
    \def\xsplit{6}   
    \def\ysplit{2.5} 
    
    \draw (\xsplit,0) -- (\xsplit,\H);      
    \draw (0,\ysplit) -- (\W,\ysplit);      
    
    \draw (0,\H) -- ++(-0.3,0) -- ++(0,-\H) -- ++(0.3,0);
    \draw (\W,\H) -- ++(0.3,0) -- ++(0,-\H) -- ++(-0.3,0);
    
    \node[above] at ($({\xsplit/2},\H)$) {$N_1$};
    \node[above] at ($({(\xsplit+\W)/2},\H)$) {$N_2$};
    
    \node[left]  at (-0.3,{\ysplit+1.6}) {$K_1$};
    \node[left]  at (-0.3,1.25) {$K_2$};
    
    \node[align=left] at ($({\xsplit/2},{\ysplit+1.6})$) {%
      row $\ell_2$-norm $=1$\\[2mm]
      column $\ell_2$-norm $\le 10^{4}\bigl(k\log n/n\bigr)^{1/2}$%
    };
    
    \node[align=left] at ($({\xsplit/2},{\ysplit-1.25})$) {%
      row $\ell_2$-norm $=1$%
    };
    
    \node[align=left] at ($({\xsplit + 0.53*(\W-\xsplit)},{\ysplit-1.25})$) {%
      each row contains at least \\
      $\ceil{250\log n}$ scales of size $\ge 100$%
    };
    
\end{tikzpicture}
    \caption{Illustration of the partition in the row rescaling $A'$ of $A$ in \cref{prop:decomp}}
    \label{fig:decomp}
\end{figure}
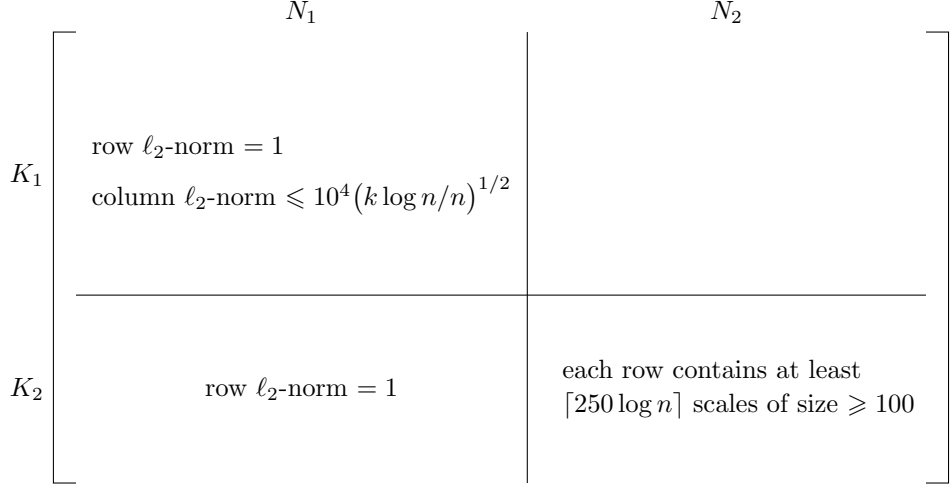

An illustration of the matrix decomposition in \cref{prop:decomp} can be found in \cref{fig:decomp}. \cref{prop:decomp} was proved in \cite{yehuda2021slicing}, but we include a proof in \cref{app:decomp} for completeness. We are now ready to present the proof of \cref{thm:main}.

\begin{proof}[Proof of \cref{thm:main}]
Let $n$ be sufficiently large, and let $\cH = \{H_1,\dots, H_k\}$ be the collection of hyperplanes. By slightly ``wiggling'' the hyperplanes, we may assume without loss of generality that every variable appears with nonzero coefficient in every hyperplane equation, and that the absolute values of all coefficients in all the hyperplane equations are distinct. Suppose for contradiction that $k\le 10^{-11}n^{13/19}\log^{-32/19}n$. Our goal is to show that there exists an edge of the hypercube $\bin^n$ that is not sliced by any of the hyperplanes in $\cH$. 

For each $i\in [k]$, we can write $H_i = \{x\in \R^n\mid \ang{a_i,x} - b_i = 0\}$ for some vector $a_i\in \R^n$ and some $b_i\in \R$. By our assumption, all entries of the vectors $a_1,\dots,a_k$ are nonzero. So we can apply \cref{prop:decomp} to the $k\times n$ matrix $A\in \R^{k\times n}$ with rows $a_1,\dots,a_k$. We obtain partitions $[k] = K_1\sqcup K_2$ and $[n] = N_1\sqcup N_2$ with $|N_2|\le n/2$ and a row rescaling $A'$ of $A$ satisfying the conditions in \cref{prop:decomp}. Since rescaling the hyperplane equations does not affect the hyperplanes and their slicing properties, we may assume without loss of generality that $A=A'$ (i.e.\ for each $i\in [k]$, we may assume that the original hyperplane equation $\ang{a_i,x} - b_i = 0$ describing $H_i$ is scaled in such a way that $a_i$ agrees with the $i$-th row of $A'$). 

Most of the work goes into the following proposition and we postpone its proof to \cref{sec:prop-proof}.

\begin{prop}\label{prop:construct-X}
For $m$ sufficiently large and $\ell\le 10^{-10}m^{13/19}\log^{-32/19}m$, consider a matrix $V\in \R^{\ell \times m}$ with the following conditions:
\begin{enumerate}
    \item For $i\in [\ell]$, the $i$-th row vector $v_i$ of $V$ satisfies $||v_i||_2 = 1$.
    \item For $j\in [m]$, the $j$-th column vector $v_{*j}$ of $V$ satisfies $||v_{*j}||_2\le 10^{-1}m^{-3/19}\log^{-13/38}m$.
\end{enumerate}
Then for any $\lambda\in \R^\ell$, there exists $X\in [-1/2, 1/2]^m$ such that the number of indices $i\in [\ell]$ satisfying $|\ang{v_i,X}-\lambda_i| \le 5\sqrt{\log m}$ is at most $\sqrt{m}/200$.
\end{prop}

\cref{prop:construct-X} shows that given a set of $\ell$ hyperplanes in $\R^m$, if the coefficient matrix given by the hyperplane equations satisfies certain conditions, then there exists a point $X\in [-1/2, 1/2]^m$ that has distance at least $5\sqrt{\log m}$ from all but at most $\sqrt{m}/200$ hyperplanes. Assuming \cref{prop:construct-X}, we have the following key lemma.

\begin{lemma}\label{lem:key-lemma}
        There exists $X\in [-1/2, 1/2]^{N_1}$ and $w\in \bin^{N_2}$ such that:
    \begin{itemize}
        \item[(i)] The number of indices $i\in K_1$ satisfying $|\ang{a_i|_{N_1},X}+\ang{a_i|_{N_2},w}- b_i| \le 4\sqrt{\log n}$ is at most $\sqrt{n}/200$. \label{item:X-K1} 

        \item[(ii)] For all $i\in K_2$, we have $|\ang{a_i|_{N_2}, w} - b_i| > 2\sqrt{n}$. \label{item:X-K2}

    \end{itemize}
\end{lemma}

\begin{proof}
We start by finding $w\in \bin^{N_2}$ satisfying (ii), this argument is identical to the proof of \cite[Claim~13]{yehuda2021slicing}. Let us consider a uniformly random vector $w\in \bin^{N_2}$. Recall that for each $i\in K_2$, by the conditions in \cref{prop:decomp}, the vector $a_i|_{N_2}$ contains at least $\ceil{250\log n}$ scales of size at least $100$. Thus, by the definition of scales, $a_i|_{N_2}$ contains at least $\ceil{250\log n}-\ceil{\log_{100}(20\sqrt{n})} \ge \ceil{200\log n}$ scales of size at least $20\sqrt{n}$. Now, by \cref{lem:many-scale-anticoncentration} we have 
\[\P[|\ang{a_i|_{N_2}, w} - b_i| \le 2\sqrt{n}] \le e^{-\ceil{200\log n} /100} \le e^{-2\log n} \le \frac{1}{n^2}.\]
Taking a union bound over all $i\in K_2$, we can conclude that there exists $w\in \bin^{N_2}$ with $|\ang{a_i|_{N_2}, w} - b_i| > 2\sqrt{n}$ for all $i\in K_2$. This means that (ii) holds.

Now, we can find the desired $X\in [-1/2, 1/2]^{N_1}$ by applying \cref{prop:construct-X} to the submatrix $A[K_1\times N_1]$ and $\lambda\in \R^{K_1}$ defined by taking $\lambda_i = b_i - \ang{a_i|_{N_2}, w}$ for all $i\in K_1$. First note that for $m:=|N_1|=n-|N_2|$ and $\ell:=|K_1|$ we have $n/2\le m\le n$ and $\ell\le k\le 10^{-11}n^{13/19}\log^{-32/19}n$. Therefore, the assumptions (1) and (2) in \cref{prop:construct-X} hold for the matrix $A[K_1\times N_1]$, because of conditions (1) and (2) in \cref{prop:decomp} (recalling that $A'=A$). In particular, for assumption (2), note that for each $j\in N_1$, we have $||a_{*j}|_{K_1}||_2\le 10^4(k\log n/n)^{1/2}\le 10^{-3/2} n^{-3/19}\log^{-13/38}n\le 10^{-1}m^{-3/19}\log^{-13/38}m$. Thus, by \cref{prop:construct-X}, we can find a point $X\in [-1/2, 1/2]^{N_1}$ such that for all but at most $\sqrt{m}/200\le \sqrt{n}/200$ indices $i\in K_1$ we have $|\ang{a_i|_{N_1},X}-\lambda_i| > 5\sqrt{\log m}$. For all of these indices $i$, we then have $|\ang{a_i,x} - b_i| = |\ang{a_i|_{N_1},X} + \ang{a_i|_{N_2},w} - b_i|=|\ang{a_i|_{N_1},X}-\lambda_i|> 5\sqrt{\log m}\ge 4\sqrt{\log n}$. Therefore (i) holds.
\end{proof}

Let us now fix $X\in [-1/2, 1/2]^{N_1}$ and $w\in \bin^{N_2}$ as in \cref{lem:key-lemma}. Let us consider a random vector $y\sim \mu_X$ and, independently from $y$, a uniformly random index $h\in N_1$. Finally, define $z\in \bin^n$ by taking $z|_{N_1}=y$ and $z|_{N_2} = w$, and let $z'\in \bin^n$ be obtained from $z$ flipping coordinate $h$. We will show that with positive probability, the hypercube edge $zz'$ is not sliced by any of the hyperplanes in $\cH$.

First, recall that in order for $zz'$ to be sliced by a hyperplane $H_i\in \cH$, the expressions $\ang{a_i,z} - b_i$ and $\ang{a_i,z'} - b_i$ must both be nonzero and have different signs. As $|(\ang{a_i,z} - b_i)-(\ang{a_i,z'} - b_i)|=2|a_{ih}|$, this in particular implies that $|\ang{a_i,z} - b_i|<2|a_{ih}|$. So it suffices to show that with positive probability we have $|\ang{a_i,z} - b_i|\ge 2|a_{ih}|$ for all $i\in [k]$.

For $i\in K_2$, we always have
\[|\ang{a_i,z} - b_i|=|\ang{a_i|_{N_1},y} +\ang{a_i|_{N_2},w} - b_i|\ge |\ang{a_i|_{N_2},w} - b_i|-|\ang{a_i|_{N_1},y}|\ge 2\sqrt{n}-\sqrt{n}>2\ge 2|a_{ih}|,\]
using (ii) and the fact that $|\ang{a_i|_{N_1},y}|\le ||a_i|_{N_1}||_2\cdot ||y||_2=1\cdot \sqrt{|N_1|}\le \sqrt{n}$ by the Cauchy-Schwarz inequality.

Let 
\[K_1^*:= \Big\{i\in K_1 \,\Big|\, |\ang{a_i|_{N_1},X}+\ang{a_i|_{N_2},w}- b_i| \le 4\sqrt{\log n}\Big\},\] 
then by (i) we have $|K_1^*|\le \sqrt{n}/200$. We first show that any hyperplane $i\in K_1\setminus K_1^*$ is very unlikely to satisfy $|\ang{a_i,z} - b_i|< 2|a_{ih}|$.

\begin{claim}\label{claim:prob-far-hyperplane-slice}
    For any $i\in K_1\setminus K_1^*$, we have 
    \[\P\Big[|\ang{a_i,z} - b_i|< 2|a_{ih}|\Big]\le \frac{2}{n^4}.\]
\end{claim}

\begin{proof}
Note that $|\ang{a_i|_{N_1},X}+\ang{a_i|_{N_2},w}- b_i| > 4\sqrt{\log n}$, since $i\not\in K_1^*$. Whenever $|\ang{a_i,z} - b_i|< 2|a_{ih}|$, we in particular have $|\ang{a_i|_{N_1},y}+\ang{a_i|_{N_2},w}- b_i|=|\ang{a_i,z} - b_i|< 2|a_{ih}|\le 2<\sqrt{\log n}$ and therefore
\[| \ang{a_i|_{N_1},X}-\ang{a_i|_{N_1},y}| \ge |\ang{a_i|_{N_1},X}+\ang{a_i|_{N_2},w}- b_i|-|\ang{a_i|_{N_1},y}+\ang{a_i|_{N_2},w}- b_i|\ge 3\sqrt{\log n}.\]
Noting that $\E[\ang{a_i|_{N_1},y}]=\ang{a_i|_{N_1},X}$ as $y\sim \mu_X$, and using \cref{lem:chernoff}, we can conclude (also recalling $||a_i|_{N_1}||_2 = 1$)
    \[
        \P\Big[|\ang{a_i,z} - b_i|< 2|a_{ih}|\Big]\le \P\Big[| \ang{a_i|_{N_1},y}-\ang{a_i|_{N_1},X}|\ge 3\sqrt{\log n}\Big] \le 2e^{-(3\sqrt{\log n})^2/2} \le 2e^{-4\log n} = \frac{2}{n^4}. \qedhere
    \]
\end{proof}

On the other hand, the following claim gives a weaker general bound on the probability for an index $i\in K_1$ (in particular, an index $i\in K_1^*$) to satisfy $|\ang{a_i,z} - b_i|< 2|a_{ih}|$.

\begin{claim}\label{claim:prob-close-hyperplane-slice}
    For any $i\in K_1$, we have
    \[\P\Big[|\ang{a_i,z} - b_i|< 2|a_{ih}|\Big]\le \frac{100}{\sqrt{n}}.\]
\end{claim}

\begin{proof}
Fix $i\in K_1$. Setting $b = b_i - \ang{a_i|_{N_2}, w}$, we have $\ang{a_i,z} - b_i=\ang{a_i|_{N_1}, y}+\ang{a_i|_{N_2}, w}-b_i=\ang{a_i|_{N_1}, y}-b$. So we have $|\ang{a_i,z} - b_i|< 2|a_{ih}|$ if and only if $|\ang{a_i|_{N_1}, y}-b|< 2|a_{ih}|$.

For $m=1,2,\dots,|N_1|$, we consider the event that $|a_{ih}|$ is the $m$-th largest entry in $a_i|_{N_1}$ in absolute value. This event  only depends on the outcome of the uniformly  random index $h\in N_1$ and has probability $1/|N_1|$. Conditioning on this event, we can apply \cref{thm:littlewood-offord} to the random vector $y\sim \mu_X$, observing that the vector $a_i|_{N_1}$ has at least $m$ entries of absolute value at least $|a_{ih}|$. Applying \cref{thm:littlewood-offord} to two intervals of size $2|a_{ih}|$ covering an interval of size $4|a_{ih}|$, we obtain
    \[\P\Big[|\ang{a_i|_{N_1}, y} - b| < 2|a_{ih}| \,\Big|\, \text{$|a_{ih}|$ is the $m$-th largest absolute value in $a_i|_{N_1}$}\Big] \le \frac{20}{\sqrt{m}}.\]
and hence
\[\P\Big[|\ang{a_i|_{N_1}, y} - b| < 2|a_{ih}| \text{ and $|a_{ih}|$ is the $m$-th largest absolute value in $a_i|_{N_1}$}\Big] \le \frac{1}{|N_1|} \cdot \frac{20}{\sqrt{m}}.\]
Summing this over all $m=1,2,\dots,|N_1|$ yields
\[\P\Big[|\ang{a_i,z} - b_i|< 2|a_{ih}|\Big]=\P\Big[|\ang{a_i|_{N_1}, y} - b| < 2|a_{ih}|\Big]\le \sum_{m=1}^{|N_1|} \frac{1}{|N_1|} \cdot \frac{20}{\sqrt{m}}=\frac{20}{|N_1|}\sum_{m = 1}^{N_1}\frac{1}{\sqrt{m}}.
        \]
Therefore, recalling that $|N_1|=n-|N_2|\ge n/2$, we can conclude
\[
       \P\Big[|\ang{a_i,z} - b_i|< 2|a_{ih}|\Big]
        \le \frac{20}{|N_1|}\Big(1+\sum_{m = 2}^{N_1}\frac{1}{\sqrt{m}}\Big)\le \frac{20}{|N_1|} \Big(1+\int_1^{|N_1|}\frac{1}{\sqrt{m}}\Big) \le \frac{20}{|N_1|}\cdot 2\sqrt{|N_1|}= \frac{40}{\sqrt{|N_1|}} \le \frac{100}{\sqrt{n}}.\qedhere
\]
\end{proof}

Combining \cref{claim:prob-far-hyperplane-slice} and \cref{claim:prob-close-hyperplane-slice}, the union bound shows that for $n$ sufficiently large, we have
\begin{align*}
    \P\Big[|\ang{a_i,z} - b_i|\ge 2|a_{ih}| \text{ for all } i\in [k]\Big]&\ge 1-\!\!\!\!\sum_{i\in K_1\setminus K_1^*}\!\frac{2}{n^4} - \sum_{i\in K_1^*}\frac{100}{\sqrt{n}}
    \ge 1-n\cdot \frac{2}{n^4} - \frac{\sqrt{n}}{200}\cdot \frac{100}{\sqrt{n}} =\frac{1}{2}-\frac{2}{n^3} >0. \qedhere
\end{align*}
\end{proof}

\section{Proof of \cref{prop:construct-X}}\label{sec:prop-proof}

In this section, we prove \cref{prop:construct-X}. Recall that $m$ is sufficiently large, $\ell\le 10^{-10}m^{13/19}\log^{-32/19}m$, and we are given an $\ell\times m$ matrix $V\in \R^{\ell \times m}$ satisfying the following conditions:
\begin{enumerate}
    \item For $i\in [\ell]$, the $i$-th row vector $v_i$ of $V$ satisfies $||v_i||_2 = 1$.
    \item For $j\in [m]$, the $j$-th column vector $v_{*j}$ of $V$ satisfies $||v_{*j}||_2\le 10^{-1}m^{-3/19}\log^{-13/38}m$.
\end{enumerate}
Also fix $\lambda\in \R^\ell$. Our goal is to show that there exists $X\in [-1/2, 1/2]^m$ such that the number of indices $i\in [\ell]$ satisfying $|\ang{v_i,X}-\lambda_i| \le 5\sqrt{\log m}$ is at most $\sqrt{m}/200$.

Geometrically, for each $i\in [\ell]$, one can consider the hyperplane $H_i = \{x\in \R^m\mid \ang{v_i,x}-\lambda_i = 0\}$ defined by the vector $v_i\in \R^m$ together with $\lambda_i$. Our goal is now to find a point $X\in [-1/2, 1/2]^m$ of  distance more than $5\sqrt{\log m}$ from all but at most $\sqrt{m}/200$ of the hyperplanes.

We will choose $X$ probabilistically and show that with positive probability it has the desired properties.

\subsection{Construction}

We construct $X = X_0+X_1$ using a two-step process defined in \cref{def:X0,def:X} below. We first sample $X_0$ according to \cref{def:X0} and then for an outcome of $X_0$, we sample $X_1$ as in \cref{def:X}.

\begin{definition}\label{def:X0}
Let $X_0\in \R^n$ be a random point defined as 
\[X_0 = \rho_0\sum_{i = 1}^\ell \alpha_i v_i, \quad \alpha_i\sim \unif\text{ for all $i\in [\ell]$},\]
where $\rho_0:= m^{3/19}\log^{-3/19}m$ and $\alpha_1,\dots,\alpha_\ell$ are independent. 
\end{definition}

In the next definition, for an outcome of $X_0$, we define an index $i\in [\ell]$ to be bad if the corresponding hyperplane $H_i = \{x\in \R^m\mid \ang{v_i,x}-\lambda_i = 0\}$ is too close to $X_0$.
It might be helpful to always think of an index $i\in [\ell]$ to be associated with the hyperplane $H_i$.

\begin{definition}[Bad indices]\label{def:bad-index}
    For an outcome of $X_0$, we say that an index $i\in [\ell]$ is \emph{bad} if 
    \[|\ang{X_0, v_i} - \lambda_i| \le 10\sqrt{\log m}.\]
\end{definition}

Now we can define $X$ as follows.

\begin{definition}\label{def:X}
Given an outcome of $X_0$, let $X_1$ be defined as
\[X_1  =\rho_1 \sum_{i \text{ bad}}\beta_i v_i, \quad \beta_i\sim U([-1,1])\text{ for all bad $i\in [\ell]$},\]
where $\rho_1 := m^{1/19}\log^{-1/19}m$ and the random variables $\beta_i$ are independent for all bad $i\in [\ell]$. Define $X = X_0+X_1$ by first sampling $X_0$ as in \cref{def:X0} and then sampling $X_1$ given the outcome of $X_0$. 
\end{definition}

Note that by \cref{def:X0} and \cref{def:X}, for any $v\in \R^m$, we have  $\E[\ang{X_0,v}] = 0$ as well as $\E[\ang{X_1,v}\mid X_0] = 0$ for any outcome of $X_0$, and consequently $\E[\ang{X,v}] = 0$. The next lemma shows that with high probability we have $X\in [-1/2,1/2]^m$.

\begin{lemma}\label{lem:X-bounded-infty-norm}
    Let $X$ be defined as in \cref{def:X}, then 
    \[\P[||X||_\infty \le 1/2] \ge 1-\frac{1}{m^2}.\]
\end{lemma}

\begin{proof}
    For $j\in [m]$, let $X(j)$, $X_0(j)$, $X_1(j)$ denote the $j$-th coordinate of $X$, $X_0$, and $X_1$, respectively. Then note that by definition
    \[X_0(j) =  \rho_0\sum_{i = 1}^\ell \alpha_i v_{ij}\]
    with independent $\alpha_i\sim \unif$ for all $i\in [\ell]$. In particular, $\E[X_0(j)] = 0$ and $X_0(j)$ is a sum of $\ell$ independent random variables in the intervals $[-\rho_0v_{ij}, \rho_0 v_{ij}]$ for $i\in [\ell]$. Since $\rho_0= m^{3/19}\log^{-3/19}m$ and $||v_{*j}||_2\le 10^{-1}m^{-3/19}\log^{-13/38}m$, we have 
    \[\sum_{i = 1}^\ell (2\rho_0 v_{ij})^2\le 4\rho_0^2\cdot ||v_{*j}||_2^2\le 4\cdot m^{6/19}\log^{-6/19}\cdot 10^{-2}m^{-6/19}\log^{-13/19}m = 4\cdot 10^{-2}\log^{-1}m.\]
    Thus by \cref{lem:chernoff}, we have 
    \[\P[|X_0(j)| > 1/3]\le 2\exp\Big(-\frac{1}{3^2}\cdot\frac{2\cdot 10^{2}\log m}{4}\Big)\le \frac{2}{m^5}.\]
    Taking a union bound over all $j\in [m]$, we have 
    \[\P[||X_0||_\infty > 1/3]\le \sum_{j=1}^{m} \P[|X_0(j)| > 1/3]\le \frac{2}{m^4}.\]
    Now for any fixed outcome of $X_0$, we have
    \[X_1(j) = \rho_1 \sum_{i \text{ bad}}\beta_i v_{ij},\]
    with independent $\beta_i\sim \unif$ for all bad $i\in [\ell]$. Therefore, $\E[X_1(j) \mid X_0] = 0$, and $X_1(j)$  conditional on the fixed outcome of $X_0$ is a sum of independent random variables in the intervals $[-\rho_1v_{ij}, \rho_1 v_{ij}]$ for all bad $i\in [\ell]$ . Since $\rho_1=m^{1/19}\log^{-1/19}m$, we have
    \[\sum_{i \text{ bad}} (2\rho_1 v_{ij})^2\le 4\rho_1^2||v_{*j}||_2^2\le 4\cdot m^{2/19}\log^{-2/19}m\cdot 10^{-2}m^{-6/19}\log^{-13/19}m = 4\cdot 10^{-2}m^{-4/19}\log^{-15/19}m.\]
    Again by \cref{lem:chernoff}, we have 
    \[\P[|X_1(j)| > 1/6 \mid X_0]\le 2\exp\Big(-\frac{1}{6^2}\cdot \frac{2\cdot 10^2m^{4/19}\log^{15/19}}{4}\Big)\le 2e^{-m^{4/19}\log^{15/19}}\le \frac{1}{m^5}.\]
    Taking a union bound over all $j\in [m]$, for any outcome of $X_0$ we have
    \[\P[||X_1||_\infty > 1/6 \mid X_0]\le \frac{1}{m^4}.\]
    Thus for $m$ sufficiently large,
    \begin{align*}
        \P[||X||_\infty \le 1/2]&\ge \P[||X_0||_\infty\le 1/3]\cdot \inf_{X_0}\P[||X_1||_\infty\le 1/6 \mid X_0]\\
        &\ge \Big(1-\frac{2}{m^4}\Big)\Big(1-\frac{1}{m^4}\Big)\ge 1-\frac{1}{m^2}.\qedhere
    \end{align*}
\end{proof}

\subsection{Analysis}

It remains to bound the probability that there are too many indices $i\in [\ell]$ such that $|\ang{v_i,X}-\lambda_i| \le 5\sqrt{\log m}$ (i.e., such that $X$ is close to the hyperplane corresponding to $H_i$).

\begin{definition}\label{def:xclose}
    For $i\in [\ell]$, we say that $i$ is \emph{close} if 
    \[|\ang{X,v_i} - \lambda_i|\le 5\sqrt{\log m}.\]
\end{definition}

The following proposition bounds the expected number of close indices. 
\begin{prop}\label{prop:expected-size-F}
    The expected number of close indices is at most $\sqrt{m}/400$.
\end{prop}

\cref{prop:construct-X} follows directly when combining \cref{lem:X-bounded-infty-norm} and \cref{prop:expected-size-F}.
\begin{proof}[Proof of \cref{prop:construct-X}]
    Let $X$ be defined as in \cref{def:X}. Then by \cref{lem:X-bounded-infty-norm}, we have $X\not\in [-1/2,1/2]^m$ with probability at most $1/m^2$. On the other hand, by \cref{prop:expected-size-F} and Markov's inequality, the probability that there are more than $\sqrt{m}/200$ close indices is at most $1/2$. Hence with probability at least $\frac{1}{2}- \frac{1}{m^2} > 0$ we have that $X\in [-1/2,1/2]^m$ and that the number of close indices is at most $\sqrt{m}/200$. Thus we can conclude that there exists $X\in [-1/2,1/2]^m$ such that the number of indices $i\in [\ell]$ satisfying $|\ang{v_i,X}-\lambda_i|\le 5\sqrt{\log m}$ is at most $\sqrt{m}/200$.
\end{proof}

The rest of the section is devoted to proving \cref{prop:expected-size-F}. For convenience, we also define for each $i\in [\ell]$ the quantity
\[
    S_i = \sum_{j\in [\ell]}\ang{v_i, v_j}^2.
\]
For each $j\in [\ell]$, recall that $\alpha_j\sim \unif$, so $\E[\alpha_j] = 0$ and $\var(\alpha_j) = 1/3$. Thus, for each $i\in [\ell]$ we have 
\begin{equation}\label{eq:var}
    \var[\ang{X_0, v_i}] = \var\Big[\rho_0\sum_{j\in [\ell]}\alpha_j\ang{v_i,v_j}\Big] = \rho_0^2\sum_{j\in [\ell]}\ang{v_i,v_j}^2\cdot \var(\alpha_j) =\frac{\rho_0^2}{3}\sum_{j\in [\ell]}\ang{v_i,v_j}^2 = \frac{1}{3}\rho_0^2S_i.
\end{equation}
Therefore, we can bound the probability of $i$ being bad using \cref{lem:continuous-littlewood-offord} as follows.
\begin{claim}\label{claim:prob-i-bad}
    For $i\in [\ell]$, 
\[\P[i \text{ is bad}] \le \frac{40\sqrt{\log m}}{\rho_0\sqrt{S_i}}.\]
\end{claim}

\begin{proof}
    Apply \cref{lem:continuous-littlewood-offord} with $X_j := \rho_0\alpha_j \ang{v_i,v_j}$ for $j\in [\ell]$. Then $X_j$ is uniformly distributed in the interval $[-\rho_0|\ang{v_i,v_j}|,\rho_0|\ang{v_i,v_j}|]$ for each $j\in [\ell]$. Furthermore, $X_1,\dots, X_\ell$ are independent, since $\alpha_1,\dots, \alpha_\ell$ are independent. Now, since $\ang{X_0, v_i} = \rho_0\sum_{j\in [\ell]}\alpha_j \ang{v_i,v_j} = \sum_{j\in [\ell]}X_j$ and $\var[\ang{X_0, v_i}] = \rho_0^2S_i/3$ by \eqref{eq:var},  \cref{lem:continuous-littlewood-offord} implies
    \[\P[i \text{ is bad}] = \P[|\ang{X_0,v_i} - \lambda_i|\le 10\sqrt{\log m}]\le \frac{20\sqrt{\log m}}{\sqrt{\rho_0^2S_i/3}}\le \frac{40\sqrt{\log m}}{\rho_0\sqrt{S_i}}. \qedhere \]
\end{proof}

In addition to bad indices, we also define the notion of \emph{near bad} indices as follows.
\begin{definition}[Near bad indices]
    For an outcome of $X_0$, we say that an index $j\in [\ell]$ is \emph{near bad} if there exists a bad index $i\in [\ell]$ such that $|\ang{v_i, v_j}| > 0.9$. 
\end{definition}
In particular, a bad index $j\in [\ell]$ is also near bad, since $|\ang{v_j, v_j}| = ||v_j||_2^2=1$.

Intuitively, a near bad index $j$ corresponds to a hyperplane $H_j$ with normal vector $v_j$ such that there exists a bad hyperplane $H_i$ with normal vector $v_i$ and the angle between $v_i$ and $v_j$ is small (i.e.\ the two hyperplanes are almost parallel).

We also need the following notion of an index $j\in [\ell]$ being heavy or light depending on $S_j$.
\begin{definition}[Heavy indices]\label{def:heavy}
    We say that an index $j\in [\ell]$ is \emph{heavy} if $S_j \ge \delta^2$ for $\delta = m^{1/38}\log^{9/19}m$; otherwise we say that $j$ is \emph{light}.
\end{definition}

Similar to the calculation in \eqref{eq:var}, we can show that for any fixed outcome of $X_0$ and any $j\in [\ell]$, we have
\[    \var[\ang{X_1, v_j}\mid X_0] = \var\Big[\rho_1\sum_{i\text{ bad}}\beta_i\ang{v_i,v_j}\Big] = \rho_1^2\sum_{i\text{ bad}}\ang{v_i,v_j}^2\cdot \var(\beta_i) =\frac{\rho_1^2}{3}\sum_{i\text{ bad}}\ang{v_i,v_j}^2.\]

To analyze the indices that are not bad given the outcome of $X_0$, but might potentially become close after sampling $X_1$, we define the following.

\begin{definition}[$t$-activated]\label{def:activated}
 Let $t = 2^h$ for some $h\in \{0,1,\dots, \ceil{\log_2 m}\}$. For an outcome of $X_0$, we say that an index $j\in [\ell]$ is \emph{$t$-activated} if the following holds:
\begin{enumerate}[label=(\arabic*)]
    \item $|\ang{X_0, v_j} - \lambda_j| \le 10t\sqrt{\log m}$, \label{item:t-activated-1}
    \item $\frac{1}{3}\rho_1^2\sum_{i \text{ bad}}\ang{v_i, v_j}^2 > t^2/4$. \label{item:t-activated-2}
\end{enumerate}
We say that $j$ is \emph{activated} if there exists $h\in \{0,1,\dots, \ceil{\log_2 m}\}$ such that for $t = 2^h$ the index $j$ is $t$-activated. For an index $j\in [\ell]$, we denote the event in (1) as $\cE_1(j,t)$ and the event in (2) as $\cE_2(j,t)$. 
\end{definition}

We can bound the probability of $\cE_1(j,t)$ using \cref{lem:continuous-littlewood-offord} in the same way as in the proof of \cref{claim:prob-i-bad} above. 

\begin{claim}\label{claim:prob-E1jt}
    For $j\in [\ell]$, $h\in \{0,1,\dots, \ceil{\log_2 m}\}$ and $t = 2^h$, we have
    \[\P[\cE_1(j,t)]\le \frac{40t\sqrt{\log m}}{\rho_0\sqrt{S_j}}.\]
\end{claim}

The following lemma states that if an index $j\in [\ell]$ is not activated and not bad, then it is unlikely that $j$ becomes close.

\begin{lemma}\label{lem:not-bad-not-activated}
    For an outcome of $X_0$, if $j\in [\ell]$ is not activated and not bad, then
    \[\P[j\text{ is close} \mid X_0]\le \frac{2}{m^2}.\]
\end{lemma}

\begin{proof} 
Fix an outcome of $X_0$ and suppose $j$ is not activated and not bad. Let $h\ge 0$ be the smallest integer such that for $t = 2^h$, we have $|\ang{X_0, v_j} - \lambda_j| \le 10t\sqrt{\log m}$. Then, since $j$ is not bad, we have $h\ge 1$ and $t\ge 2$. In particular, by the minimality of $h$, we can conclude $|\ang{X_0, v_j} - \lambda_j| > 10\cdot 2^{h-1}\sqrt{\log m}=5t\sqrt{\log m}$.

We claim that $\frac{1}{3}\rho_1^2\sum_{i\text{ bad}}\le t^2/4$. If $h\le \ceil{\log_2 m}$, this follows from the fact that $j$ is not $t$-activated. If $h>\ceil{\log_2 m}$ (and hence $t>m$), this simply follows because $\frac{1}{3}\rho_1^2\sum_{i \text{ bad}}\ang{v_i, v_j}^2 \le \frac{1}{3}\rho_1^2 \ell \le m^2/4<t^2/4$.

If $j$ is close, we must have $|\ang{X_1, v_j}| \ge |\ang{X_0, v_j} - \lambda_j|-|\ang{X, v_j} - \lambda_j|> (5t-5)\sqrt{\log m}\ge 2t\sqrt{\log m}$. On the other hand, we have $4\rho_1^2\sum_{i\text{ bad}}\ang{v_i,v_j}^2 \le 12(t^2/4) = 3t^2$. Thus, by \cref{lem:chernoff}, we have
\[
 \P[j\text{ is close} \mid X_0] 
\le \P\Big[|\ang{X_1, v_j}| > 2t\sqrt{\log m} \,\Big|\, X_0\Big]
\le 2e^{-2(2t\sqrt{\log m})^2/3t^2}\le \frac{2}{m^{2}}.\qedhere
\]

\end{proof}

We also compute the probability that $j$ is close and activated, which will be useful. 

\begin{lemma}\label{lem:prob-activated-in-F}
    For $j\in [\ell]$ we have 
    \[\P[j\text{ is close and activated}]\le \frac{1600\log^2 m}{\rho_0\sqrt{S_j}}.\]
\end{lemma}

\begin{proof}
    For any $h = 0,\dots, \ceil{\log_2 m}$, setting $t=2^h$, by \cref{claim:prob-E1jt} we have
    \begin{align*}
        \P[j\text{ is close and $t$-activated}] &\le \P[j \text{ is $t$-activated}]\cdot \P[j\text{ is close}\mid j \text{ is $t$-activated}]\\
        &\le \P[\cE_1(j,t)]\cdot \P[j\text{ is close}\mid j \text{ is $t$-activated}]\\
        &\le \frac{40t\sqrt{\log m}}{\rho_0\sqrt{S_j}}\cdot \P[j\text{ is close}\mid j \text{ is $t$-activated}].
    \end{align*}
    To bound the second factor $\P[j\text{ is close}\mid j \text{ is $t$-activated}]$, we fix an outcome of $X_0$ so that $j$ is $t$-activated. Let $\lambda_j' = \lambda_j - \ang{X_0, v_j}$, and note that then $\ang{X,v_j} - \lambda_j=\ang{X_0+X_1,v_j} - \lambda_j=\ang{X_1,v_j} - \lambda_j'$. Therefore, by \cref{lem:continuous-littlewood-offord} and the fact that $\var[\ang{X_1,v_j}\mid X_0]=\frac{1}{3}\rho_1^2\sum_{i \text{ bad}}\ang{v_i, v_j}^2 > t^2/4$, we have
    \[\P[j\text{ is close}\mid X_0]\le \P\Big[\ang{X_1,v_j} - \lambda_j'|\le 5\sqrt{\log m} \,\Big|\, X_0\Big]\le \frac{10\sqrt{\log m}}{\sqrt{t^2/4}}=\frac{20\sqrt{\log m}}{t}\]
    for any fixed outcome of $X_0$ such that $j$ is $t$-activated. So we can conclude that
    \[\P[j\text{ is close}\mid j \text{ is $t$-activated}]\le \frac{20\sqrt{\log m}}{t}\]
    and hence
    \[\P[j\text{ is close and $t$-activated}] \le \frac{40t\sqrt{\log m}}{\rho_0\sqrt{S_j}}\cdot\frac{20\sqrt{\log m}}{t} =  \frac{800\log m}{\rho_0\sqrt{S_j}}.\]
    Now, taking a union bound over $h = 0,\dots, \ceil{\log_2 m}$, for $m$ sufficiently large we have 
    \[
        \P[j\text{ is close and activated}]\le \sum_{h = 0}^{\ceil{\log_2 m}} \P[j\text{ is close and $2^h$-activated}]\le \sum_{h = 0}^{\ceil{\log_2 m}} \frac{800\log m}{\rho_0\sqrt{S_j}}
        \le \frac{1600\log^2 m}{\rho_0\sqrt{S_j}}.\qedhere
    \]
\end{proof}

The rest of the section will be devoted to bounding the expected number of close indices $j\in [\ell]$ by considering five different cases, as listed in the following claim.

\begin{claim}\label{claim:close-categories}
    For an outcome of $X_0$, every index $j\in [\ell]$ falls into (at least) one of the following types:
    \begin{enumerate}
    \item $j$ is bad;
    \item $j$ is not bad and not activated;
    \item $j$ is activated and heavy;
    \item $j$ is light and near bad;
    \item $j$ is activated and not near bad.
\end{enumerate}
\end{claim}

\begin{proof}
    Fix an outcome of $X_0$, and let $j\in [\ell]$. If $j$ is bad then it falls into (1), so we may assume $j$ is not bad. Then if $j$ is not activated, it falls into (2), so we may assume $j$ is activated. If $j$ is heavy, then it falls into (3), so we can assume $j$ is light. If $j$ is near bad, then it falls into (4). Thus, finally if $j$ is not near bad, it falls into (5).
\end{proof}

Note that the expected number of close indices $j$ of type (2), i.e.\ where $j$ is not bad and not activated, can be bounded using \cref{lem:not-bad-not-activated}. The other four types will be covered in the following four subsections.

\subsubsection{$j$ bad}

The following lemma can be used to bound the expected number of close indices that are bad.

\begin{lemma}\label{lem:prob-bad-close}
    For $j\in [\ell]$, we have 
    \[\P[j \text{ is bad and close}]\le \frac{800\log m}{\rho_0\rho_1}.\]
\end{lemma}

\begin{proof}
    For an outcome of $X_0$, if $j\in [\ell]$ is bad, then we have 
    \[\var[\ang{X_1,v_j}\mid X_0] = \frac{1}{3}\rho_1^2\sum_{i\text{ bad}}\ang{v_i,v_j}^2\ge \frac{1}{3}\rho_1^2\ang{v_j,v_j}^2=\frac{1}{3}\rho_1^2.\]
    Thus, defining $\lambda_j' := \lambda_j - \ang{X_0,v_j}$ and applying \cref{lem:continuous-littlewood-offord}, we have 
    \[\P[j\text{ is close}\mid X_0] \le \P\Big[|\ang{X_1,v_j} - \lambda_j'|\le 5\sqrt{\log m}\,\Big|\, X_0\Big]\le \frac{10\sqrt{\log m}}{\sqrt{\rho_1^2/3}}\le \frac{20\sqrt{\log m}}{\rho_1}\]
    for any outcome of $X_0$ such that $j$ is bad. This implies that
    \[\P[j\text{ is close}\mid j\text{ is bad}] \le  \frac{20\sqrt{\log m}}{\rho_1}.\]
    Now, by \cref{claim:prob-i-bad} and the fact that $S_j\ge 1$, we have
    \[\P[j \text{ is bad and close}] = \P[j\text{ is bad}]\cdot \P[j\text{ is close}\mid j\text{ is bad}]\le \frac{40\sqrt{\log m}}{\rho_0}\cdot \frac{20\sqrt{\log m}}{\rho_1} = \frac{800\log m}{\rho_0\rho_1}.\qedhere\]
\end{proof}

\subsubsection{$j$ activated and heavy}

We bound the expected number of close indices that are heavy and activated. Recall that an index $j\in [\ell]$ is heavy if $S_j \ge \delta^2$.

\begin{lemma}\label{lem:num-heavy-in-F}
    The expected number of close indices that are heavy and activated is at most $1600\ell\log^2 m/(\rho_0\delta)$. 

\end{lemma}

\begin{proof}
    Suppose $j\in [\ell]$ is heavy, then by \cref{lem:prob-activated-in-F} and the fact that $S_j \ge \delta^2$, we have
    \[\P[j\text{ is close and activated}] \le \frac{1600\log^2 m}{\rho_0\delta}.\]
    Summing over all heavy indices $j\in [\ell]$ proves the lemma.
\end{proof}

\subsubsection{$j$ light and near bad}

Recall that $j\in [\ell]$ is light if $S_j< \delta^2$, and that $j$ is near bad if there exists a bad index $i\in [\ell]$ such that $|\ang{v_i,v_j}| > 0.9$. In the following lemma, we bound the expected number of indices that are light and near bad.

\begin{lemma}\label{lem:num-light-near-bad}
    The expected number of indices that are light and near bad is at most $100\delta \ell\sqrt{\log m}/\rho_0$.
\end{lemma}

\begin{proof}
    For convenience, we consider the graph $G$ defined as follows: the vertex set is $[\ell]$ and two vertices $i,j\in [\ell]$ are connected by an edge if $|\ang{v_i, v_j}| > 0.9$. For $i\in [\ell]$, let $\deg(i)$ denote the degree of vertex $i$, and let $N(i)$ denote the neighborhood of vertex $i$.
    
    Note that if $|\ang{v_i, v_j}| > 0.9$ for some $i,j\in [\ell]$, then we have either $||v_i-v_j||_2^2\le 0.2$ or $||v_i+v_j||_2^2\le 0.2$. Indeed, since $||v_i||_2 = 1$ and $||v_j||_2 = 1$, we have
    \[||v_i-v_j||_2^2 = ||v_i||_2^2 + ||v_j||_2^2 - 2\ang{v_i,v_j} = 2-2\ang{v_i,v_j}.\]
    Similarly, we have $||v_i+v_j||_2^2 = 2+2\ang{v_i,v_j}$. So, depending on the sign of $\ang{v_i,v_j}$, we have either $||v_i-v_j||_2^2\le 0.2$ or $||v_i+v_j||_2^2\le 0.2$.

    \begin{claim}\label{claim:light-bad-deg-bound}
        For every light index $j\in [\ell]$ that is near bad, there exists a bad neighbor $i\in N(j)$ such that $\deg(i)\le 3\delta^2$.
    \end{claim}

    \begin{proof}
        Let $j\in [\ell]$ be a light index that is near bad. Then by definition there exists $i\in N(j)$ that is bad. For any $h\in N(i)$, since $|\ang{v_i, v_h}| > 0.9$, we have either $||v_i- v_h||_2^2\le 0.2$ or $||v_i+v_h||_2^2\le 0.2$. We claim that $|\ang{v_j, v_h}|\ge 0.6$ for all $h\in  N(i)$. Indeed, assuming without loss of generality that $||v_i- v_j||_2^2\le 0.2$ and $||v_i- v_h||_2^2\le 0.2$,  by the triangle inequality we have
        \[||v_j-v_h||_2\le ||v_j- v_i||_2 + ||v_i- v_h||_2 \le \sqrt{0.2} + \sqrt{0.2} = 2\sqrt{1/5}.\]
        So we have $||v_j-v_h||_2^2\le 0.8$ and therefore $|\ang{v_j, v_h}|\ge (2-0.8)/2 = 0.6$. Now we obtain
        \[0.6^2\cdot |N(i)| \le\sum_{h\in N(i)} \ang{v_j, v_h}^2 \le \sum_{h = 1}^k \ang{v_j, v_h}^2 = S_j\le \delta^2,\]
        recalling the assumption that $j$ is light. Thus we have $\deg(i)=|N(i)|\le 3\delta^2$, as claimed.
    \end{proof}
    
By \cref{claim:light-bad-deg-bound}, the number of indices $j\in [\ell]$ that are light and near bad is at most 
\[\sum_{\substack{i\in [\ell]\text{ bad}\\\deg(i)\le 3\delta^2}}\deg(i).\]
Therefore, the expected number of indices that are light and near bad is at most
\[\sum_{\substack{i\in [\ell]\\ \deg(i)\le 3\delta^2}} \deg(i)\cdot \P[i\text{ is bad}] \le \sum_{\substack{i\in [\ell]\\ \deg(i)\le 3\delta^2}} \deg(i)\cdot \frac{40\sqrt{\log m}}{\rho_0 \sqrt{S_i}},\]
using \cref{claim:prob-i-bad}. Now, note that for every $i\in [\ell]$, we have $S_i=\sum_{j\in [\ell]} \ang{v_i, v_j}^2\ge \sum_{j\in N(i)} \ang{v_i, v_j}^2\ge \deg(i)\cdot 0.9^2\ge \deg(i)/2$. Therefore we can conclude that the expected number of indices that are light and near bad is at most
\[\sum_{\substack{i\in [\ell]\\ \deg(i)\le 3\delta^2}}\deg(i)\cdot \frac{40\sqrt{\log m}}{\rho_0 \sqrt{\deg(i)/2}}
= \frac{40\sqrt{\log m}}{\rho_0}\cdot \sum_{\substack{i\in [\ell]\\ \deg(i)\le 3\delta^2}}\sqrt{2\deg(i)}\le \frac{40\sqrt{\log m}}{\rho_0 }\cdot \ell\cdot \sqrt{6}\delta \le \frac{100\delta \ell\sqrt{\log m}}{\rho_0}. \qedhere\]
\end{proof}

Thus we have the following upper bound for the expected number of close indices that are light and near bad.

\begin{lemma}\label{lem:num-near-bad-light-in-F}
    The expected number of close indices that are light and near bad is at most $2000\delta \ell\log m/(\rho_0\rho_1)$.
\end{lemma}

\begin{proof}
Consider any light index $j\in [\ell]$. For any outcome of $X_0$ such that $j$ is near bad, there exists at least one bad index $i\in [\ell]$ satisfying $|\ang{v_i, v_j}| > 0.9$, and consequently
\[\var[\ang{X_1, v_j}\mid X_0] = \frac{1}{3}\cdot\rho_1^2\sum_{i\text{ bad}}\ang{v_i, v_j}^2\ge \frac{1}{3}\cdot 0.9^2\rho_1^2.\]
Thus, by \cref{lem:continuous-littlewood-offord} we have
 \[\P[j\text{ is close} \mid X_0]\le \frac{10\sqrt{\log m}}{0.9\rho_1/\sqrt{3}}\le \frac{20\sqrt{\log m}}{\rho_1}\]
 for every outcome of $X_0$ such that $j$ is near bad. Therefore, we can conclude that
 \[\P[j\text{ is close and near bad}]=\P[j\text{ is close}\mid j\text{ is near bad}]\cdot \P[j\text{ is near bad}]\le \frac{20\sqrt{\log m}}{\rho_1}\cdot \P[j\text{ is near bad}]\]
 for any light index $j\in [\ell]$. Thus, the expected number of close indices that are light and near bad is
 \[\sum_{\substack{j\in [\ell]\\j\text{ light}}}\P[j\text{ is close and near bad}]\le \frac{20\sqrt{\log m}}{\rho_1}\cdot \sum_{\substack{j\in [\ell]\\j\text{ light}}}\P[j\text{ is near bad}].\]
 Note that the sum on the right-hand side is precisely the expected number of indices that are light and near bad. Plugging in \cref{lem:num-light-near-bad} to bound this, we can conclude that the expected number of close indices that are light and near bad is at most
     \[\frac{20\sqrt{\log m}}{\rho_1}\cdot \frac{100\delta \ell\sqrt{\log m}}{\rho_0} = \frac{2000\delta \ell\log m}{\rho_0\rho_1}. \qedhere\]
\end{proof}

\subsubsection{$j$ activated and not near bad}

Finally, we bound the expected number of close indices that are activated and not near bad. Recall from \cref{def:activated} that for $j\in [\ell]$ and $t\in \{1,2,4,\dots,2^{\ceil{\log_2 m}}\}$, $\cE_1(j,t)$ denotes the event that $|\ang{X,v_j} - \lambda_j|\le 10t\sqrt{\log m}$. We first prove the following key lemma bounding the probability that $i$ is bad and $\cE_1(j,t)$ holds for distinct $i,j\in [\ell]$ satisfying $|\ang{v_i,v_j}|\le 0.9$.

\begin{lemma}\label{lem:almost-indep}
    For distinct $i,j\in [\ell]$ with $|\ang{v_i, v_j}|\le 0.9$ and $t\in \{1,2,4,\dots,2^{\ceil{\log_2 m}}\}$, we have 
    \[\P[i\text{ bad}\text{ and }\cE_1(j,t)]\le  \frac{40000t}{\rho_0^2\sqrt{S_j}}\log m.\]
\end{lemma}
    
\begin{proof}
We first show that for any given outcomes of $\alpha_h$ for all $h\in [\ell]\setminus\{i,j\}$, we have
\begin{equation}\label{eq:conditional-on-alpha-h}
\P[i\text{ bad}\text{ and }\cE_1(j,t) \mid (\alpha_h)_{h\in [\ell]\setminus\{i,j\}}]\le  \frac{2000t}{\rho_0^2}\log m.
\end{equation}

To show \eqref{eq:conditional-on-alpha-h}, let us fix arbitrary outcomes of $\alpha_h$ for all $h\in [\ell]\setminus\{i,j\}$. Let $A = \ang{v_i, v_j}$. By assumption we have $|A|\le 0.9$. Note that 
\[\ang{X_0, v_i} = \rho_0\Big(\alpha_i + \alpha_j\cdot A + \sum_{h \in [\ell]\setminus\{i,j\}}\alpha_h\ang{v_h, v_i}\Big)\]
and 
\[\ang{X_0, v_j} = \rho_0\Big(\alpha_j + \alpha_i\cdot A + \sum_{h \in [\ell]\setminus\{i,j\}}\alpha_h\ang{v_h, v_j}\Big),\]
and that the values of $\sum_{h \in [\ell]\setminus\{i,j\}}\alpha_h\ang{v_h, v_i}$ and $\sum_{h \in [\ell]\setminus\{i,j\}}\alpha_h\ang{v_h, v_j}$ are already determined by the fixed outcomes of $\alpha_h$ for $h\in [\ell]\setminus\{i,j\}$.

In the event $\cE_1(j,t)$, we have
\[|\ang{X_0,v_j} - \lambda_j|\le 10t\sqrt{\log m}.\]
So, the random variable $\alpha_j + \alpha_i\cdot A$ must be contained in the interval $I_j$ of length $20t\sqrt{\log m}/\rho_0$ centered around $\lambda_j/\rho_0 - \sum_{h \in [\ell]\setminus\{i,j\}}\alpha_h\ang{v_h, v_j}$. On the other hand, if $i$ is bad, we have 
\[|\ang{X_0, v_i} - \lambda_i| \le 10\sqrt{\log m},\]
so the random variable $\alpha_i + \alpha_j\cdot A$ must be contained in the interval $I_i$ of length $20\sqrt{\log m}/\rho_0$ centered around $\lambda_i/\rho_0 - \sum_{h \in [\ell]\setminus\{i,j\}}\alpha_h\ang{v_h, v_i}$. 

Thus, if $\cE_1(j,t)$ holds and $i$ is bad, taking the linear combination $(\alpha_j + A\cdot \alpha_i) - A(\alpha_i + A\cdot \alpha_j) = (1- A^2 )\alpha_j$, we see that $(1- A^2 )\alpha_j$ must be contained in the interval $I_j-AI_i$ of length at most $40t\sqrt{\log m}/\rho_0$ and so $\alpha_j$ must be contained in a particular interval of length at most $40t\sqrt{\log m}/((1-A^2)\rho_0)\le 400t\sqrt{\log m}/\rho_0$ (namely, the interval $(1- A^2 )^{-1}(I_j-AI_i)$). Here we used the inequality $1-A^2\ge 1/10$, which follows from $|A|\le 0.9$.

The probability that $\alpha_j\sim U[-1,1]$ lies in this interval is at most $200t\sqrt{\log m}/\rho_0$. Now conditioning on any such outcome of $\alpha_j$, for the event of $i$ being bad to occur, $\alpha_i\sim U[-1,1]$ must fall into an interval of length $20\sqrt{\log m}/\rho_0$ (which happens with probability at most $10\sqrt{\log m}/\rho_0$). Thus we can conclude that 
\[\P[i\text{ bad}\text{ and }\cE_1(j,t)\mid (\alpha_h)_{h\in [\ell]\setminus\{i,j\}}]\le  \frac{200t}{\rho_0}\sqrt{\log m}\cdot \frac{10}{\rho_0}\sqrt{\log m}=\frac{2000t}{\rho_0^2}\log m,\]
showing \eqref{eq:conditional-on-alpha-h}.

In the case where $S_j\le 400$, the desired inequality in the statement of the lemma follows directly from \eqref{eq:conditional-on-alpha-h}. Indeed, in this case we automatically have
\[\P[i\text{ bad}\text{ and }\cE_1(j,t)]\le  \frac{2000t}{\rho_0^2}\log m\le \frac{40000t}{\rho_0^2\sqrt{S_j}}\log m.\]

So let us assume $S_j > 400$ from  now on. We distinguish between the cases $\rho_0\le 10t\sqrt{\log m}$ and $\rho_0 > 10t\sqrt{\log m}$. In either case, in the event $\cE_1(j,t)$, we must have 
\[|\ang{X_0, v_j} - \lambda_j| = \Big|\rho_0\Big(\alpha_i+\alpha_j\cdot A + \sum_{h\in [\ell]\setminus\{i,j\}}\alpha_h\ang{v_h,v_i} \Big)- \lambda_j\Big| \le 10t\sqrt{\log m}.\]
So the sum $\rho_0\sum_{h\in [\ell]\setminus\{i,j\}}\alpha_h\ang{v_h,v_i}$ must lie in the interval of length $20t\sqrt{\log m}+4\rho_0$ centered at $\lambda_j$.

If $\rho_0\le 10t\sqrt{\log m}$, then this interval has length $20t\sqrt{\log m}+4\rho_0 \le 60t\sqrt{\log m}$. By \cref{lem:continuous-littlewood-offord}, the probability that $\rho_0\sum_{h\in [\ell]\setminus\{i,j\}}\alpha_h\ang{v_h,v_i}$ falls into this interval is at most 
\[\frac{60t\sqrt{\log m}}{\rho_0\sqrt{S_j - 1 - A^2}/\sqrt{3}}\le \frac{\sqrt{3}\cdot 60t\sqrt{\log m}}{\rho_0\sqrt{S_j - 2}}\le \frac{200t\sqrt{\log m}}{\rho_0\sqrt{S_j}}.\]
Conditioning on any such outcomes of $\alpha_h$ for all $h\in [\ell]\setminus\{i,j\}$, as well as on an outcome of $\alpha_j$, the event of $i$ being bad occurs only when $\alpha_i$ falls into a particular interval of length $20\sqrt{\log m}/\rho_0$, and the probability for this is at most $10\sqrt{\log m}/\rho_0$. So we can conclude that 
\[\P[i \text{ bad and }\cE_1(j,t)]\le \frac{200t\sqrt{\log m}}{\rho_0\sqrt{S_j}}\cdot \frac{10\sqrt{\log m}}{\rho_0} = \frac{2000t\log m}{\rho_0^2\sqrt{S_j}}.\]

If $\rho_0 > 10t\sqrt{\log m}$, in the event $\cE_1(j,t)$ the sum $\rho_0\sum_{h\in [\ell]\setminus\{i,j\}}\alpha_h\ang{v_h,v_i}$ must fall into the interval of length $20t\sqrt{\log m}+4\rho_0\le 6\rho_0$ centered at $\lambda_j$. By \cref{lem:continuous-littlewood-offord}, this happens with probability at most
\[\frac{6\rho_0}{\rho_0\sqrt{S_j - 1 - A^2}/\sqrt{3}}\le \frac{\sqrt{3}\cdot 6}{\sqrt{S_j - 2}}\le \frac{20}{\sqrt{S_j}}.\]
Conditioning on any such outcomes of $\alpha_h$ for all $h\in [\ell]\setminus\{i,j\}$, by \eqref{eq:conditional-on-alpha-h} the probability that $i$ is bad and $\cE_1(j,t)$ holds is at most  $(2000t/\rho_0^2)\log m$. So we can conclude 
\[\P[i\text{ bad}\text{ and }\cE_1(j,t)]\le  \frac{20}{\sqrt{S_j}}\cdot\frac{2000t}{\rho_0^2}\log m=\frac{40000t}{\rho_0^2\sqrt{S_j}}\log m. \qedhere\]
\end{proof}

Now, for an index $j\in [\ell]$, we can bound the probability of $j$ being activated and not near bad.

\begin{lemma}\label{lem:activated-not-near-bad}
    For $j\in [\ell]$, we have
    \[\P[j\text{ is activated and not near bad}]\le 120000\cdot \frac{\rho_1^2}{\rho_0^2}\cdot\log m\sqrt{S_j}.\]
\end{lemma}

\begin{proof}
     First, for any $h = 0,\dots, \ceil{\log_2 m}$, we bound the probability of $j$ being $t$-activated for $t=2^h$ and not near bad. For convenience, define the quantity
    \[T_j =\frac{1}{3} \cdot \rho_1^2 \sum_{\substack{i\in [\ell] \text{ bad}\\ |\ang{v_i, v_j}|\le 0.9}}\ang{v_i,v_j}^2 = \frac{1}{3} \cdot\rho_1^2\sum_{\substack{i\in [\ell]\\ |\ang{v_i, v_j}|\le 0.9}}\ang{v_i,v_j}^2\cdot 1_{i\text{ bad}},\]
    where $1_{i\text{ bad}}$ is the indicator variable for the event of $i$ being bad. Now, if $j$ is not near bad, then $T_j = \frac{1}{3} \cdot\rho_1^2\sum_{i\in [\ell]\text{ bad}}\ang{v_i,v_j}^2$ agrees with the sum in the event $\cE_2(j,t)$. Therefore, if $j$ is $t$-activated and not near bad, we must have $T_j > t^2/4$. Thus, we have
    \[\P[j\text{ is $t$-activated and not near bad}] \le \P[\cE_1(j,t)\text{ and }T_j > t^2/4] = \P[T_j\cdot 1_{\cE_1(j,t)} > t^2/4].\]
    By \cref{lem:almost-indep}, the expectation of $T_j\cdot 1_{\cE_1(j,t)}$ can be computed as
    \begin{align*}
        \E[T_j\cdot 1_{\cE_1(j,t)}] &=\frac{1}{3} \cdot\rho_1^2 \sum_{\substack{i\in [\ell]\\|\ang{v_i,v_j}|\le 0.9}}\ang{v_i,v_j}^2\cdot\E[1_{i\text{ bad}}\cdot 1_{\cE_1(j,t)}]\\
        &=\frac{1}{3} \cdot\rho_1^2 \sum_{\substack{i\in [\ell]\\|\ang{v_i,v_j}|\le 0.9}}\ang{v_i,v_j}^2\cdot \P[i\text{ bad}\text{ and }\cE_1(j,t)]\\
        &\le \frac{1}{3} \cdot\rho_1^2 \sum_{\substack{i\in [\ell]\\|\ang{v_i,v_j}|\le 0.9}}\ang{v_i, v_j}^2\cdot \frac{40000t}{\rho_0^2\sqrt{S_j}}\log m\\
        &\le \frac{15000t\rho_1^2}{\rho_0^2\sqrt{S_j}}\log m \sum_{i\in [\ell]}\ang{v_i, v_j}^2= \frac{15000t\rho_1^2}{\rho_0^2\sqrt{S_j}}\log m\cdot S_j
        = \frac{\rho_1^2 }{\rho_0^2}\cdot 15000t\log m\sqrt{S_j}.
    \end{align*}
    Thus by Markov's inequality, we have
    \[
        \P[j\text{ is $t$-activated and not near bad}] = \P[T_j\cdot 1_{\cE_1(j,t)} > t^2/4]\le \frac{4}{t^2}\cdot \E[T_j\cdot 1_{\cE_1(j,t)}]\le 60000\cdot\frac{\rho_1^2}{\rho_0^2}\cdot \frac{\log m\sqrt{S_j}}{t}.
    \]
     Now, taking a union bound, we have
    \begin{align*}
        \P[j\text{ is activated and not near bad}] &\le \sum_{h = 0}^{\ceil{\log_2 m}} \P[j\text{ is $2^h$-activated and not near bad}]\\
        &\le \sum_{h = 0}^{\ceil{\log_2 m}} 60000 \cdot\frac{\rho_1^2}{\rho_0^2}\cdot \frac{\log m\sqrt{S_j}}{2^h}
        \le 120000\cdot \frac{\rho_1^2}{\rho_0^2}\log m \sqrt{S_j}. \qedhere
    \end{align*}
\end{proof}

Now we are ready to bound the number of close indices that are activated and not near bad.

\begin{corollary}\label{cor:not-near-bad-in-F}
    The expected number of close indices that are activated and not near bad is at most $10^5\cdot \rho_1\ell\log^{3/2} m/\rho_0^{3/2}$.
\end{corollary}

\begin{proof}
    For each $j\in [\ell]$, by \cref{lem:activated-not-near-bad} and \cref{lem:prob-activated-in-F}, we have
    \begin{align*}
        &\P[j\text{ is close, activated and not near bad}]\\
        &\qquad\le \min \Big(\P[j\text{ is activated and not near bad}],\, \P[j\text{ is close and activated}]\Big)\\
        &\qquad\le \min \Big(120000\cdot \frac{\rho_1^2}{\rho_0^2}\log m \sqrt{S_j},\, \frac{1600\log^2 m}{\rho_0\sqrt{S_j}}\Big)\\
        &\qquad\le \Big(120000\cdot \frac{\rho_1^2}{\rho_0^2}\log m \sqrt{S_j}\cdot \frac{1600\log^2 m}{\rho_0\sqrt{S_j}}\Big)^{1/2}\\
        &\qquad\le 10^5\cdot \frac{\rho_1}{\rho_0^{3/2}}\log^{3/2} m.
    \end{align*}
    Summing over all $j\in [\ell]$, it follows that the expected number of close indices that are activated and not near bad is at most $10^5\cdot \rho_1\ell\log^{3/2} m/\rho_0^{3/2}$.
\end{proof}

\subsection{Proof of \cref{prop:expected-size-F}}

Combining the estimates proved previously, we are now ready to prove \cref{prop:expected-size-F}.

\begin{proof}[Proof of \cref{prop:expected-size-F}]
We separately bound the expected number of close indices $j\in [\ell]$ for each of the five types (1) to (5) listed in \cref{claim:close-categories}. Recall that we defined the relevant parameters as follows:

    \begin{itemize}
        \item $\rho_0 = m^{3/19}\log^{-3/19}m$,
        \item $\rho_1 = m^{1/19}\log^{-1/19}m$,
        \item $\delta = m^{1/38}\log^{9/19}m$,
    \end{itemize}
    and recall our assumption $\ell\le 10^{-10}m^{13/19}\log^{-32/19}m$.

    By \cref{lem:prob-bad-close}, the expected number of close indices $j\in [\ell]$ of type (1), i.e.\ the expected number of close indices that are bad, is at most 
    \[\frac{800\log m}{\rho_0\rho_1} \cdot \ell\le 8\cdot 10^{-8}m^{9/19}\log^{-9/19}m\le 8\cdot 10^{-8}\sqrt{m}.\]
    
    By \cref{lem:not-bad-not-activated}, the expected number of close indices $j\in [\ell]$ of type (2), i.e.\ the expected number of close indices that are not bad and not activated, is at most
    \[\frac{2}{m^2}\cdot \ell\le \frac{2}{m}\le 10^{-8}\sqrt{m}.\]

    By \cref{lem:num-heavy-in-F}, the expected number of close indices $j\in [\ell]$ of type (3), i.e.\ the expected number of close indices that are activated and heavy, is at most
    \begin{align*}
        \frac{1600\ell\log^2 m}{\delta\rho_0}\le 2\cdot 10^{-7}\sqrt{m}. 
    \end{align*}
    
    By \cref{lem:num-near-bad-light-in-F}, the expected number of close indices $j\in [\ell]$ of type (4), i.e.\ the expected number of close indices that are light and near bad, is at most
    \begin{align*}
        \frac{2000\delta \ell\log m}{\rho_0\rho_1}\le 2\cdot 10^{-7}\sqrt{m}.
    \end{align*}
    
    By \cref{cor:not-near-bad-in-F}, the expected number of close indices $j\in [\ell]$ of type (5), i.e.\ the expected number of close indices that are activated and not near bad, is at most 
    \begin{align*}
         10^5\cdot \frac{\rho_1}{\rho_0^{3/2}}\ell\log^{3/2} m \le 10^{-5}\sqrt{m}.
    \end{align*}
    
    Summing the bounds above yields that the expected number of close indices is at most $\sqrt{m}/400$.
\end{proof}

\section{Acknowledgments}

The second author would like to thank Marcelo Campos for inspiring discussions.

\printbibliography

\appendix

\input{appendix}

\end{document}

%% file: appendix.tex
\section{Proofs of probabilistic lemmas}\label{app:probability-proofs}

In this appendix, for the reader's convenience, we collect the proofs of the probabilistic lemmas in \cref{sec:prelim}. The following proof is given in \cite{yehuda2021slicing}.

\begin{proof}[Proof of \cref{lem:many-scale-anticoncentration}]

We first need the following claim.

\begin{claim}[\cite{yehuda2021slicing}]\label{claim:scale-magnitude-interval}
    Let $u\in \R^n$ be a vector with $||u||_2 = 1$, then for  uniformly random $x\in\bin^n$, we have
    \[\P\Big[\frac{1}{5}\le |\ang{x,u}|\le 5\Big] \ge \frac{1}{5}.\]
\end{claim}

\begin{proof}
    Let $Z:= \ang{x,u}^2$ and note that $\E[Z] = ||u||_2^2 = 1$. Moreover, we have 
    \[\E[Z^2] = \sum_{i = 1}u_i^4 + 6\sum_{i\ne j}u_i^2u_j^2 \le 3\Big(\sum_{i = 1}^n u_i^2\Big)^2 = 3||u||_2^4 = 3 \]
    By Markov's inequality, we obtain
    \[\P[Z\ge 5^2]\le \frac{\E[Z]}{5^2}=\frac{1}{25}.\]
    On the other hand, by the Paley-Zygmund inequality, we have 
    \[\P\Big[Z\ge \frac{1}{5^2}\Big]=\P\Big[Z\ge \frac{\E[Z]}{25}\Big]\ge \Big(1-\frac{1}{25}\Big)^2\cdot \frac{\E[Z]^2}{\E[Z^2]}\ge \Big(1-\frac{1}{25}\Big)^2\cdot \frac{1}{3}\ge \frac{1}{4}.\]
    Thus we have
    \[\P\Big[\frac{1}{5}\le |\ang{x,u}|\le 5\Big]\ge \P\Big[Z\ge \frac{1}{5^2}\Big]- \P[Z\ge 5^2] \ge \frac{1}{4} - \frac{1}{25}\ge \frac{1}{5}. \qedhere\]
\end{proof}

Now we are ready to prove the lemma. Let $I_1,\dots, I_s\subseteq [n]$ be disjoint subsets as in \cref{def:scales}, satisfying $||v|_{I_s}||_2 \ge 10\delta$ and $||v|_{I_i}||_2 \ge 100||v|_{I_{i+1}}||_2$ for $i\in [s-1]$. Furthermore, let $y\sim \bin^n$ be a uniformly random vector in $\bin^n$, and let $\eps_1,\dots, \eps_s\sim \bin$ be independent Rademacher random variables (also independent from $y$). Finally, let $x\in \bin^n$ be the vector defined as follows: for $j\in [n]$, if there exists $h\in [s]$ with $j\in I_h$, then define $x_j = y_j\cdot \eps_h$; otherwise $x_j = y_j$. Note that $x$ is uniformly distributed in $\bin^n$.

For convenience let $v^{(i)} = v|_{I_i}$ and $y^{(i)} = y|_{I_i}$ for all $i\in [s]$. Note that since $I_1,\dots, I_s$ are disjoint, $y^{(1)},\dots, y^{(s)}$ are independent random vectors. Let $Z_i = \ang{y^{(i)}, v^{(i)}}$ for every $i\in [s]$ and note that $Z_1,\dots, Z_s$ are independent. For each $i\in [s]$, by \cref{claim:scale-magnitude-interval}, with probability at least $1/5$ we have
\begin{equation}\label{eq:scale-interval}
    \frac{||v^{(i)}||_2}{5}\le |Z_i| \le 5||v^{(i)}||_2.
\end{equation}
Thus, by the Chernoff bound, with probability at least $1 - e^{-s/20}$ at least $s/60$ indices $i\in [s]$ satisfy \eqref{eq:scale-interval}. 

Let $R\subseteq [s]$ denote the set of indices $i\in [s]$ satisfying \eqref{eq:scale-interval}. Then we have $\P[|R|<s/60 ]\le e^{-s/20}$. Note that for $r,r'\in R$ with $r < r'$, we have 
\[|Z_r|\ge \frac{||v^{(r)}||_2}{5}\ge 20||v^{(r')}||_2\ge 4|Z_{r'}|\ge \frac{||v^{(r')}||_2}{5} \ge 2\delta.\]

\begin{claim}\label{claim:unique-assignment-many-scale}
    For any fixed outcome of $y\in \bin^n$, and any fixed outcomes of $\eps_h$ for $h\in [s]\setminus R$, there is at most one possible outcome for $\eps|_R\in \bin^R$ satisfying $|\ang{x,v} - b| \le \delta$.
\end{claim}

\begin{proof}
Given all the fixed outcomes, $\ang{x,v}-b$ is an affine-linear function in the entries of $\eps|_R$. For each $r\in R$, the coefficient of $\eps_r$ in this affine-linear function is precisely $Z_r = \ang{y^{(r)}, v^{(r)}}$.

    Suppose for contradiction that there are two distinct outcomes for $\eps|_R\in \bin^R$ such that the absolute value of this affine-linear function is at most $\delta$. Denoting the entries of these outcomes by $\eps_r$ and $\eps'_r$ for $r\in R$, we must then have
    \[\Big|\sum_{r\in R}(\eps_r - \eps_r')Z_r\Big|=\Big|\sum_{r\in R}Z_r\eps_r - \sum_{r\in R}Z_r\eps'_r\Big|\le 2\delta.\]
    On the other hand, taking the minimum index $q\in R$ with $\eps_q\ne \eps_q'$, we have
    \[\Big|\sum_{r\in R}(\eps_r - \eps_r')Z_r\Big|\ge |\eps_q - \eps_q'|\cdot |Z_q|-\Big|\sum_{\substack{r\in R\\r>q}}(\eps_r - \eps_r')Z_r\Big|\ge 2|Z_q|-2\sum_{\substack{r\in R\\r>q}}|Z_r|\ge 2|Z_q|-2\sum_{\substack{r\in R\\r>q}}\frac{|Z_q|}{4^{r-q}}\ge \frac{4}{3}|Z_{q}| \ge \frac{8\delta}{3}.\]
    Here, the third inequality uses the fact that $|Z_r|\ge 4|Z_{r'}|$ for all $r,r'\in R$ with $r < r'$, and the last inequality uses the fact that $|Z_{q}|\ge 2\delta$. This gives a contradiction and finishes the proof of the claim.
\end{proof}

    By \cref{claim:unique-assignment-many-scale}, conditioning on any given outcome of $R$, we have $\P[|\ang{x,v} - b| \le \delta\mid R ]\le 2^{-|R|}$. In particular, this gives
    \[\P\Big[|\ang{x,v} - b| \le \delta \,\Big|\, |R|\ge s/60\Big]\le 2^{-s/60}.\]
    Thus we can conclude 
    \[
        \P[|\ang{x,v} - b|\le \delta]\le \P\Big[|\ang{x,v} - b| \le \delta \,\Big|\, |R|\ge s/60\Big] + \P[|R| < s/60]\le 2^{-s/60} + e^{-s/20} \le e^{-s/100},
    \]
    where the last inequality uses the assumption $s \ge 100$.
\end{proof}

The following proof is standard in the literature and adapted from \cite{klein2022slicing}.

\begin{proof}[Proof of \cref{thm:littlewood-offord}]
    First consider the case where $p = (0,\dots, 0)$, then $x$ is uniformly distributed in $\bin^n$. Without loss of generality, we may assume $v_i\ge 0$ for all $i \in [n]$. Then the set of $x\in \bin^n$ satisfying $|\ang{x,v} - b| < t$ forms an antichain in the $m$-dimensional subcube generated by $\{j\in [n] \mid |v_j|\ge t\}$. Thus by Sperner's theorem and Stirling's approximation, we have
    \[\P[|\ang{x,v} - b| < t] \le \frac{1}{2^m}\cdot\binom{m}{\floor{m/2}} \le  \frac{1}{2^m}\cdot\frac{e}{\pi}\cdot\frac{2^m}{\sqrt{m}}\le \frac{1}{\sqrt{m}}.\]

    Now we reduce the general case (for any $p$ with $||p||_\infty \le 1/2$) to the case $p = (0,\dots, 0)$. Let $I_+\subseteq [n]$ be the set of indices $i\in [n]$ with $p_i>0$, and let $I_-\subseteq [n]$ be the set of indices $i\in [n]$ with $p_i<0$. We now sample the random vector $x\sim \mu_p$ in a multi-step procedure as follows. First, select independent random subsets $J_+\subseteq I_+$ and $J_-\subseteq I_-$ such that independently each $i\in I_+$ is contained in $J_+$ with probability $p_i$, and independently each $i\in I_-$ is contained in $J_-$ with probability $-p_i$. Then, let us define $x_j=1$ for all $j\in J_+$ and $x_j=-1$ for all $j\in J_-$. Finally, let $S=[n]\setminus (J_+\cup J_-)$, and choose $x|_{S}\sim \bin^{S}$ uniformly at random. It is not difficult to check that the resulting overall distribution of $x$ is precisely the distribution $\mu_p$.

    Let $m'$ be the number of indices $i\in S=[n]\setminus (J_+\cup J_-)$ with $|v_i|\ge t$. Conditioning on any outcomes of $J_+$ and $J_-$, and setting $b'=b+\sum_{j\in J_-}v_j-\sum_{j\in J_+}v_j$, we have
    \[\P[|\ang{x,v} - b| < t\mid J_+,J_-]=\P[|\ang{x|_S,v|_S} - b'| < t\mid J_+,J_-]\le \frac{1}{\sqrt{m'}}\]
    by the bound for the case $p = (0,\dots, 0)$ applied to the uniformly random vector $x|_{S}\sim \bin^{S}$. In particular, for any outcomes of $J_+$ and $J_-$ such that $m'\ge m/4$, this probability is at most $1/\sqrt{m/4}$. Therefore we obtain
    \[\P[|\ang{x,v} - b| < t\mid m'\ge m/4]\le \frac{1}{\sqrt{m/4}}=\frac{2}{\sqrt{m}}\]
     As $||p||_\infty \le 1/2$, every index is contained in $J_+\cup J_-$ with probability at most $1/2$. Therefore we have $\E[m']\ge m/2$, and consequently by Chebyshev's inequality,
    \[\P[m' < m/4]\le \P[m'-\E[m'] \ge m/4]\le \frac{\var[m']}{(m/4)^2}\le \frac{m/4}{m^2/16}= \frac{4}{m}.\]
    Thus, we have
    \[\P_{x\sim \mu_p}[|\ang{x,v} - b| < t] \le \P[|\ang{x,v} - b| < t\mid m'\ge m/4]+\P[m' < m/4]\le \frac{2}{\sqrt{m}}+\frac{4}{m}\le \frac{10}{\sqrt{m}}.\qedhere\]  
\end{proof}

The proof of \cref{lem:continuous-littlewood-offord} follows directly from a more general statement proved by Bobkov and Chistyakov~\cite{BobkovChistyakov2015-ConcentrationFunctions}. Recall that a random variable is log-concave if the logarithm of its probability density function is concave. In particular, the uniform distribution on a continuous real interval is log-concave.

\begin{prop}[\cite{BobkovChistyakov2015-ConcentrationFunctions}]\label{prop:log-concave-anticoncentration}
    If $X_1,\dots, X_n$ are real-valued independent log-concave random variables, then for $X = X_1 + \dots + X_n$ and any $t > 0$, we have
    \[\sup_{x\in \R}\P[X\in [x, x+t]] \le \frac{t}{\sqrt{\Var(X)+(t^2/12)}}.\]
\end{prop}

\begin{proof}[Proof of \cref{lem:continuous-littlewood-offord}]
    Apply the proposition to $X = \sum_{i = 1}^n X_i$ with $X_i \sim U[a_i, b_i]$ for $i\in [n]$ and note that since $X_i$ is uniform on a bounded interval, it is log-concave. Thus, we have
    \[\P[|X - b| < t]\le \sup_{x\in \R}\P[X\in [x,x+2t]]\le \frac{2t}{\sqrt{\Var(X)+(t^2/12)}}\le \frac{2t}{\sqrt{\var(X)}}.\qedhere\]
\end{proof}

\section{Proof of the decomposition}\label{app:decomp}

The following proof is given in \cite{yehuda2021slicing}.

\begin{proof}[Proof of \cref{prop:decomp}]
    For notational convenience, let $S:= \ceil{250\log n}$ and $W:= 10^4(k\log n/n)^{1/2}$. We want to find partitions $[k] = K_1\sqcup K_2$ and $[n] = N_1\sqcup N_2$ with $|N_2|\le n/2$ as well as a row rescaling $A'$ of $A$ satisfying the following conditions:
    \begin{enumerate}
    \item For every $i\in [k]$, the $i$-th row vector $a'_i\in \R^n$ of $A'$ satisfies $||a'_i|_{N_1}||_2 = 1$.

    \item For every $j\in N_1$, the $j$-th column vector $a'_{*j}\in \R^k$ of $A'$ satisfies $||a'_{*j}|_{K_1}||_2\le  W$.

    \item For every $i\in K_2$, the restriction $a_i'|_{N_2}$ of the $i$-th row vector $a_i'$ of $A'$ to $N_2$ contains at least $S$ scales of size at least $100$.
\end{enumerate}

    We use the following algorithm that starts with all the column indices in $N_1$ and all the row indices in $K_1$. Throughout the algorithm, we move column indices $j\in N_1$ with $||a'_{*j}|_{K_1}||>W$ from $N_1$ to $N_2$ one by one. As we move column indices from $N_1$ to $N_2$, if $||a'_i|_{N_1}||_2$ drops below a certain threshold for some row index $i\in K_1$, then we rescale the $i$-th row such that $||a'_i|_{N_1}||_2=1$. This rescaling will lead to a ``new scale'' of size at least $100$ in the vector $a'_i|_{N_2}$. We will move a row index $i\in K_1$ into $K_2$ when the corresponding row has been rescaled at least $S$ times (then $a'_i|_{N_2}$ will have at least $S$ scales of size at least $100$). 

    Let $\tau = 1/10001$ so that 
    \[\sqrt{\frac{1-\tau}{\tau}} = 100.\]
    Recall that by assumption, all entries of $A$ are nonzero. The algorithm is defined as follows.
    \begin{itemize}
        \item[1.] Initialize $N_1 = [n]$, $N_2 = \varnothing$, $K_1 = [k]$, $K_2 = \varnothing$. Initialize $A'$ as a row rescaling of $A$ such that for all $i=1,\dots,k$, the $i$-th row  $a'_i$ of $A$ satisfies $||a'_i|_{N_1}||_2 = 1$.
        \item[2.] While there exists an index $j\in N_1$ such that $\sum_{i\in K_1}(a'_{ij})^2\ge \tau W^2$:
        \begin{enumerate}
            \item[a.] Move $j$ from $N_1$ to $N_2$.
            \item[b.] For all $i\in K_1$ with $\sum_{j\in N_1}(a'_{ij})^2\le \tau$, renormalize the $i$-th row in $A'$ such that $||a'_i|_{N_1}||_2 = 1$.
            \item[c.] For all $i\in K_1$ such that the $i$-th row of $A'$ has been renormalized at least $S$ times throughout the algorithm, move $i$ from $K_1$ to $K_2$. 
        \end{enumerate}
        \item[3.] Renormalize every row in $A'$ such that $||a'_i|_{N_1}||_2 = 1$. Terminate and output the partitions $[n] = N_1\sqcup N_2$ and $[k] = K_1\sqcup K_2$ as well as the matrix $A'$.
    \end{itemize}
    We now show that the outputs of the algorithm indeed satisfy the conditions (1) to (3) above. Condition (1) clearly holds by the renormalization in Step 3 before outputting $A'$. 
    
    To see that condition (2) holds, first note that at the end of every iteration of the loop in Step 2, we have $\sum_{j\in N_1}(a'_{ij})^2> \tau$ and hence $||a'_i|_{N_1}||_2 >\sqrt{\tau}$ for each $i\in K_1$ (otherwise, row $i$ would have been renormalized in Step 2b). In particular, this means that in the renormalization in Step 3, for each $i\in K_1$, the $i$-th row is multiplied by a factor of at most $1/\sqrt{\tau}$. Furthermore, note that at the start of Step 3, we have $||a'_{*j}|_{K_1}||_2^2=\sum_{i\in K_1}(a'_{ij})^2\le \tau W^2$ for all $j\in N_1$ (otherwise, Step 2 would continue). This means that at the end of Step 3 we have $||a'_{*j}|_{K_1}||_2\le (1/\sqrt{\tau})\cdot \sqrt{\tau W^2}=W$, which establishes (2).

    To check condition (3), let us fix an index $i\in [k]$ and assume that $i\in K_2$ at the end of the algorithm. This means that the $i$-th row of $A'$ has been renormalized at least $S$ times in Step 2b. Let $N_1(1)\supseteq \dots \supseteq N_1(S)$ be the sets $N_1$ at the times where row $i$ was renormalized in Step 2b, and let $N_1(0)=[k]$. Note that for the final set $N_1$ at the end of the algorithm, we then have $N_1\subseteq N_1(S)$.
    
    We claim that $||a'_i|_{N_1(h)}||_2^2 \,/\,||a'_i|_{N_1(h-1)}||_2^2\le \tau $ for $h=1,\dots,S$. To see this, first note that the ratio $||a'_i|_{N_1(h)}||_2^2\, /\,||a'_i|_{N_1(h-1)}||_2^2$ is invariant under rescaling of the rows. Just before Step 2b is applied at the time where $N_1$ equals $N_1(h)$, we have $||a'_i|_{N_1(h)}||_2^2=\sum_{j\in N_1(h)}(a'_{ij})^2\le \tau$ (due to the fact that we are about to renormalize row $i$ in Step 2b), and $||a'_i|_{N_1(h-1)}||_2^2=1$ (from the previous renormalization or initial normalization of row $i$). Thus, we indeed have $||a'_i|_{N_1(h)}||_2^2\, /\,||a'_i|_{N_1(h-1)}||_2^2\le \tau $ at that time and therefore also at the end of the process. 

    Thus, for $h=1,\dots,S$, for the output matrix $A'$ we have
    \[\frac{||a'_i|_{N_1(h-1)\setminus N_1(h)}||_2^2}{||a'_i|_{N_1(h)}||_2^2}=\frac{\sum_{j\in N_1(h-1)\setminus N_1(h)}(a'_{ij})^2}{\sum_{j\in N_1(h)}(a'_{ij})^2}=\frac{\sum_{j\in N_1(h-1)}(a'_{ij})^2}{\sum_{j\in N_1(h)}(a'_{ij})^2}-1=\frac{||a'_i|_{N_1(h-1)}||_2^2}{||a'_i|_{N_1(h)}||_2^2}-1\ge \frac{1}{\tau}-1\]
    and therefore
    \[\frac{||a'_i|_{N_1(h-1)\setminus N_1(h)}||_2}{||a'_i|_{N_1(h)}||_2}\ge \sqrt{\frac{1}{\tau}-1}=\sqrt{\frac{1-\tau}{\tau}}=100.\]
    This means that $||a'_i|_{N_1(h-1)\setminus N_1(h)}||_2\ge 100 ||a'_i|_{N_1(h)}||_2\ge 100 ||a'_i|_{N_1(h)\setminus N_1(h+1)}||_2$ for $h=1,\dots,S-1$ and $||a'_i|_{N_1(S-1)\setminus N_1(S)}||_2\ge 100 ||a'_i|_{N_1(S)}||_2\ge 100 ||a'_i|_{N_1}||_2=100$ (where the equality sign is due to the final renormalization in Step 3). Thus, taking the disjoint sets $I_h=N_1(h-1)\setminus N_1(h)\subseteq [n]\setminus N_1=N_2$ for $h=1,\dots,S$ in \cref{def:scales}, we can conclude that the vector $a_i'|_{N_2}$ has at least $S$ scales of size at least $100$. 

    Now it only remains to show that $|N_2|\le n/2$. To this end, we study how the sum
    \begin{equation}\label{eq:qunatity-sum-track}
       \sum_{i\in K_1}\sum_{j\in N_1}(a'_{ij})^2 
    \end{equation}
    evolves throughout the algorithm. After Step 1, this sum is $\sum_{i\in [k]} ||a'_i||_2^2=k$. Every time a row gets renormalized in Step 2b, the sum in \eqref{eq:qunatity-sum-track} increases by at most 1 (and there is no other step in which it can increase). Thus, the total increase of \eqref{eq:qunatity-sum-track} throughout the whole algorithm is at most $kS$ (since each of the rescalings in Step 2b happens at most $S$ times for every row). On the other hand, every time an index $j$ is moved from $N_1$ to $N_2$,  \eqref{eq:qunatity-sum-track} decreases by at least $\tau W^2$. Thus, we have
    \[|N_2|\cdot \tau W^2 \le k+k S=k(S+1),\]
    which implies that $|N_2|\le k(S+1)/(\tau W^{2})\le 10001k(S+1)/W^{2}$. Substituting in the appropriate parameters $S= \ceil{250\log n}$ and $W= 10^4(k\log n/n)^{1/2}$, we can conclude that $|N_2|\le n/2$, as desired.
\end{proof}